\documentclass[11pt]{article}
\usepackage{amsfonts}
\usepackage{amssymb}
\usepackage{amssymb,amsfonts,amsmath,amsthm,cite,color}
\usepackage{dsfont}
\usepackage{epsfig}
\usepackage{mathrsfs}
\usepackage{algorithmic}
\usepackage{graphicx}
\usepackage{epstopdf}
\usepackage{setspace}
\usepackage{caption}
\usepackage{graphicx, subfig}
\usepackage{tikz}
\usepackage{titlesec}
\usepackage{makecell}
\usepackage{enumitem}
\titlespacing{\section}{0pt}{6pt}{4pt}
\titlespacing{\subsection}{0pt}{4pt}{2pt}
\titleformat{\section}
  {\normalfont \bfseries \large}
  {\thesection}
  {1em}
  {}

\titleformat{\subsection}
  {\normalfont \bfseries \normalsize}
  {\thesubsection}
  {1em}
  {}

\newtheorem{thm}{Theorem}[section]
\newtheorem{conj}[thm]{Conjecture}
\newtheorem{cor}[thm]{Corollary}
\newtheorem{lem}[thm]{Lemma}
\newtheorem{prop}[thm]{Proposition}

\newtheorem{pro}[thm]{Problem}

\newtheorem{obs}[thm]{Observation}

\makeatletter \@addtoreset{equation}{section}

\begin{document}
\begin{center}
\begin{spacing}{1.5}
{\Large \bf Saturation numbers of $K_{2}\vee P_{k}$} \\[3pt]
\end{spacing}
\end{center}

\begin{center}
{Xiaoxue Zhang}, {Lihua You$^*$}, {Xinghui Zhao}

School of Mathematical Sciences\\
South China Normal University, Guangzhou 510631, P.R. China\\[10pt]
zhang\_xx1209@163.com, ylhua@scnu.edu.cn, xhzhao@m.scnu.edu.cn\\
\end{center}

\let\thefootnote\relax\footnotetext{$^*$Corresponding author.}

\vskip 3mm \noindent {\bf Abstract:}
A graph $G$ is called $H$-saturated if $G$ contains no copy of $H$, but $G+e$ contains a copy of $H$ for any edge $e\in E(\overline{G})$. The saturation number of $H$ is the minimum number of edges in an $H$-saturated graph of order $n$, denoted by $sat(n,H)$. In this paper, we investigate $sat(n,K_{2}\vee P_{k})$, where $k\geq 3$. Let $a_k$ be an integer, defined as follows: $a_k=k$ for $3\leq k\leq 5$; $a_k=3\cdot 2^{t-1}-2$ for $k=2t\geq 6$; and $a_k=2^{t+1}-2$ for $k=2t+1\geq 7$. We show that $sat(n, K_{2}\vee P_{k})=2n-3+sat(n-2,P_{k})$ for $n\geq a_k+2$ and $k\geq 3$, characterize the $K_{2}\vee P_{k}$-saturated graphs with $sat(n,K_{2}\vee P_{k})$ edges, the $K_{1}\vee P_{k}$-saturated graphs with $sat(n,K_{1}\vee P_{k})$ edges for $3\leq k\leq5$ and the $P_{k}$-saturated graphs with $sat(n, P_{k})$ edges for $3\leq k\leq4$.
Furthermore, we propose some questions for further research.


\section{Introduction}
All graphs considered in this paper are simple, finite and undirected.
We introduce some terminologies and notations, and refer to \cite{Bondy} for terminologies and notations not defined in the paper.

Let $G=(V(G),E(G))$ be a graph with vertex set $V(G)$ and edge set $E(G)$. The \emph{order} of $G$ is $|V(G)|$ and the \emph{size} of $G$ is $|E(G)|$. A path, complete graph, or cycle of order $n$ is denoted by $P_{n}$, $K_{n}$ or $C_{n}$, respectively.
For a vertex $u\in V(G)$, the set of neighbours of $u$ in $G$ is denoted by $N_{G}(u)$, $d_{G}(u)=|N_{G}(u)|$ is the \emph{degree} of  $u$ in $G$, and $N_{G}[u]=N_{G}(u)\cup \{u\}$ is the closed neighborhood of $u$ in $G$.
For $u,v\in V(G)$, the \emph{distance} between $u$ and $v$ is the length of the shortest path connecting $u$ and $v$ in $G$, and denoted $d_{G}(u,v)$.
The \emph{diameter} of $G$ is the greatest distance between two vertices of $G$, denoted by $diam(G)$.
When no ambiguity arises, we use $N(u)$, $d(u)$, $N[u]$ and $d(u,v)$ to replace $N_{G}(u)$, $d_{G}(u)$, $N_{G}[u]$ and $d_{G}(u,v)$, respectively.
Let $\delta(G)$ and $\Delta(G)$ be the minimum degree and maximum degree of vertices in $G$, respectively.
A vertex of degree zero is called an \emph{isolated vertex}.
For $u\in V(G)$, we call $u$ a \emph{conical vertex} of $G$ if $d(u)=|V(G)|-1$.
An \emph{independent set} in a graph is a set of vertices no two of which are adjacent.
For $X\subseteq V(G)$, we use $G[X]$ to denote the subgraph of $G$ induced by $X$. For $X\subseteq V(G)$ and $Y\subseteq V(G)\backslash X$, $E_{G}[X,Y]$ is the set of edges of $G$ between $X$ and $Y$, denoted simply by $E[X,Y]$.
The \emph{union} of two graphs $G_{1}$ and $G_{2}$ is the graph $G_{1}\cup G_{2}$ with vertex set $V(G_{1})\cup V(G_{2})$ and edge set $E(G_{1})\cup E(G_{2})$.
The \emph{join} of graphs $G_{1}$ and $G_{2}$ is obtained from $G_{1}\cup G_{2}$ by joining each vertex of $G_{1}$ to each vertex of $G_{2}$, denoted by $G_{1}\vee G_{2}$.

A graph $H$ is called a \emph{subgraph} of a graph $G$ if $V(H)\subseteq V(G)$ and $E(H)\subseteq E(G)$. A \emph{copy} of $H$ in $G$ is a subgraph of $G$ which is isomorphic to $H$.
A graph $G$ is called \emph{$H$-saturated} if $G$ contains no copy of $H$, but $G+e$ contains a copy of $H$ for any edge $e\in E(\overline{G})$. The \emph{saturation number} of $H$ is the minimum number of edges in an $H$-saturated graph of order $n$, denoted by $sat(n,H)$. If $G$ is an $H$-saturated graph with $sat(n,H)$ edges, we say $G$ is a \emph{minimal $H$-saturated graph}.

The saturation number, as a concept opposite to the Tur\'{a}n number, was first introduced by Erd\"{o}s, Hajnal and Moon \cite{Erdos} in $1964$, though they did not use this terminology. Meanwhile, they proved that $sat(n,K_{p})=(p-2)(n-p+2)+\binom{p-2}{2}$ and $K_{p-2}\vee \overline{K}_{n-p+2}$ is the unique minimal $K_{p}$-saturated graph. In $1986$, K\'{a}szonyi and Tuza \cite{Kaszonyi} obtained a general upper bound for $sat(n,\mathcal{F})$, where $\mathcal{F}$ is a family of graphs, which implies that $sat(n,\mathcal{F})=O(n)$.
Since then, many scholars have focused on this field and determined the saturation numbers for specific graphs, such as stars, paths and matchings \cite{Kaszonyi}, cycles \cite{Ollmann,Tuza,Chen2009,Chen2011,Lan,Furedi}, complete bipartite graphs \cite{chen2014,Huang}, linear forests \cite{chen2015,Fan,Cao}, vertex-disjoint cliques \cite{Faudree,Chen2024,Zhu}, books and generalized books \cite{ChenandFaudree} and so on. For more results, we refer the reader to the survey \cite{Currie}.

The join of graphs plays a fundamental role in graph theory, and the saturation numbers of the join of certain special graphs have been determined. In $2024$, Hu, Luo and Peng \cite{Hu} showed that, if $D$ is a graph without any isolated vertex, then $sat(n,K_{s}\vee D)=\binom{s}{2}+s(n-s)+sat(n-s,D)$ for $n\geq 3s^2-s+2sat(n-s,D)+1$ and $s\geq 1$. However, Qiu, He, Lu and Xu \cite{Qiu} pointed out that $n\geq 3s^2-s+2sat(n-s,D)+1$ implies that $D$ has a component $K_2$. Meanwhile, they also obtained $sat(n,K_{1}\vee C_{k})=n-1+sat(n-1,C_{k})$ for $k\geq 8$ and $n\geq 56k^3$ in \cite{Qiu}. When $k=4$, Song, Hu, Ji and Cui \cite{Song} proved that $sat(n,K_{1}\vee C_{4})=n-1+sat(n-1,C_{4})$ for $n\geq6$. In $2025$, Hu, Ji and Cui \cite{Hu2025} proved that $sat(n, K_{1}\vee P_{k})=n-1+sat(n-1,P_{k})$ for $k\geq 5$ and $n$ sufficiently large.

Based on the above results, the study of saturation number of the join of $K_{2}$ and $P_{k}$ is meaningful.
When $k=2$, we have $K_{2}\vee P_{2}\cong K_{4}$, according to the results of Erd\"{o}s, Hajnal and Moon \cite{Erdos}, $sat(n,K_{4})=2n-3$ for $n\geq 4$ and unique minimal $K_{4}$-saturated graph is $K_{2}\vee \overline{K}_{n-2}$.

In this paper, by using different methods from that used in $K_{1}\vee P_{k}$-saturated graphs, we determine the value of $sat(n,K_{2}\vee P_{k})$ and the minimal $K_{2}\vee P_{k}$-saturated graphs for $k\geq 3$, characterize the minimal $K_{1}\vee P_{k}$-saturated graphs for $3\leq k\leq5$ and the minimal $P_{k}$-saturated graphs for $3\leq k\leq4$. The following is our main result.

\begin{thm}\label{thm1-1}
Let $k\geq 3$, $n\geq a_k+2$, where \begin{equation}\label{eq1.1}
a_{k} =
\begin{cases}
k,  & \text{if $3\leq k\leq 5$}; \\
3 \cdot 2^{t-1}-2, & \text{if $k\geq 6$ and $k=2t$}; \\
2^{t+1}-2, & \text{if $k\geq 7$ and $k=2t+1$}.
\end{cases}
\end{equation}
Then
\begin{equation}\label{eq1.2}
sat(n, K_{2}\vee P_{k})=2n-3+sat(n-2,P_{k})=
\begin{cases}
\lfloor \frac{5n-8}{2} \rfloor,  & \text{if $k=3$}; \\
\frac{5n-5}{2}, & \text{if $k=4$, $n$ is odd}; \\
\frac{5n-8}{2}, & \text{if $k=4$, $n$ is even}; \\
\lceil \frac{17n-32}{6} \rceil, & \text{if $k=5$}; \\
3n-5-\lfloor \frac{n-2}{a_{k}} \rfloor, & \text{if $k\geq 6$}.
\end{cases}
\end{equation}
If $G$ is a $K_{2}\vee P_{k}$-saturated graph of order $n$, then $G$ is a minimal $K_{2}\vee P_{k}$-saturated graph if and only if $G\cong K_{2}\vee F$ for $n\geq10$, and $G\in\mathcal{G}$ for $n\leq9$, where $F$ is a minimal $P_{k}$-saturated graph of order $n-2$, $H_i$ $(1\leq i\leq10)$ are shown in Figure $1$, and
\vspace{-0.3cm}
\[\mathcal{G}=
\begin{cases}
\{K_{2}\vee F,K_1\vee H_1,K_1\vee H_2,K_1\vee H_3,H_9\},  & \text{if $k=3$}; \\
\{K_{2}\vee F,K_1\vee H_4,K_1\vee H_5,K_1\vee H_6,H_{10}\}, & \text{if $k=4$}; \\
\{K_{2}\vee F,K_1\vee H_7,K_1\vee H_8\}, & \text{if $k=5$}.
\end{cases}\]
\end{thm}

The rest of this paper is organized as follows. In Section $2$, we present some known results that are useful for the proof. In Section $3$, we study $sat(n,K_1\vee P_k)$ and the minimal $K_1\vee P_k$-saturated graphs for $3\leq k\leq5$. In Section $4$, we prove Theorem \ref{thm1-1}. In Section $5$, we study the minimal $P_{k}$-saturated graphs for $3\leq k\leq4$, and propose some problems for further research. In Section $6$, we conclude this paper.

\begin{center}
\scalebox{1.2}[1.2]{\includegraphics{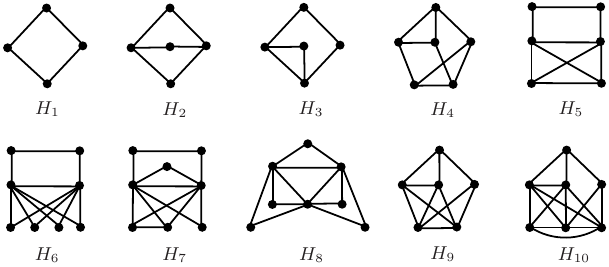}}\\
\captionof{figure}{Graphs $H_{i}$, where $1\leq i\leq 10$.}
\end{center}

\section{Preliminaries}
In this section, we present some results that will be used later.

For $P_{k}$-saturated graphs and $K_{1}\vee P_{k}$-saturated graphs, the authors \cite{Kaszonyi}, \cite{Hu2025} obtained the following results.

\begin{thm}[\rm\!\!\cite{Kaszonyi}]\label{thm2-1}
Let $k\geq 3$ and $a_{k}$ defined as $\mathrm{(\ref{eq1.1})}$. Then we have
\mbox{}\\
$\mathrm{(i)}$$sat(n,P_{3})=\lfloor \frac{n}{2} \rfloor$ for $n\geq 3$.\vspace{0.1cm}\\
$\mathrm{(ii)}$$sat(n,P_{4})=
\begin{cases}
\frac{n}{2}, & \text{for $n\geq 4$, $n$ is even};\\
\frac{n+3}{2},  & \text{for $n\geq 5$, $n$ is odd}.
\end{cases}$\vspace{0.1cm}\\
$\mathrm{(iii)}$$sat(n,P_{5})=\lceil \frac{5n-4}{6} \rceil$ for $n\geq 5$.\vspace{0.1cm}\\
$\mathrm{(iv)}$$sat(n,P_{k})=n-\lfloor \frac{n}{a_{k}} \rfloor$ for $n\geq a_{k}$ and $k\geq 6$.
\end{thm}

\begin{thm}[\rm\!\!\cite{Hu2025}]\label{thm2-2}
Let $a_{k}$ defined as $\mathrm{(\ref{eq1.1})}$, with $n\geq
\begin{cases}
7, & \text{if $k=5$};\\
a_{k}+1,  & \text{if $k\geq 6$}.
\end{cases}$\\
Then $sat(n, K_{1}\vee P_{k})=n-1+sat(n-1,P_{k})$.
Moreover, the minimal $K_{1}\vee P_{k}$-saturated graph is either $K_{1}\vee F$, where $F$ is a minimal $P_{k}$-saturated graph of order $n-1$, or $H_7$, $H_8$ for $k=5$.
\end{thm}

The following lemma is useful and interesting.

\begin{lem}\label{lem2-3}
Let $G$ be a graph of order $n$, $u\in V(G)$ with $d_{G}(u)=n-1$. Then $G$ is $K_{1}\vee H$-saturated if and only if $G[V(G)\backslash\{u\}]$ is $H$-saturated.
\end{lem}

\begin{proof}
Suppose that $G$ is $K_{1}\vee H$-saturated, then $G[V(G)\backslash\{u\}]$ contains no copy of $H$. Otherwise, $G$ contains a copy of $K_{1}\vee H$, a contradiction. For any $e\in E(\overline{G}[V(G)\backslash\{u\}])\subseteq E(\overline{G})$, $G+e$ contains a copy of $K_{1}\vee H$, say $\hat{H}$. If $u\notin V(\hat{H})$, then $G[V(G)\backslash\{u\}]$ contains a copy of $H$, a contradiction. Therefore, $u\in V(\hat{H})$, and $u$ is a conical vertex of $\hat{H}$ by $d_{G}(u)=n-1$. Thus $G[V(G)\backslash\{u\}]+e$ contains a copy of $H$, and $G[V(G)\backslash\{u\}]$ is $H$-saturated.

Suppose that $G[V(G)\backslash\{u\}]$ is $H$-saturated. Then $G[V(G)\backslash\{u\}]$ contains no copy of $H$, thus $G$ contains no copy of $K_{1}\vee H$. For any $e\in E(\overline{G})$, we know $e\in E(\overline{G}[V(G)\backslash\{u\}])$ by $d_{G}(u)=n-1$, then $G[V(G)\backslash\{u\}]+e$ contains a copy of $H$ since $G[V(G)\backslash\{u\}]$ is $H$-saturated, and thus $G+e$ contains a copy of $K_{1}\vee H$. Therefore, $G$ is $K_{1}\vee H$-saturated.
\end{proof}

By Lemma \ref{lem2-3}, the following result is obtained immediately.

\begin{cor}[\rm\!\!\cite{Cameron}]\label{cor2-4}
For all $n\geq |V(H)|+1$, $sat(n,K_{1}\vee H)\leq n-1+sat(n-1,H)$.
\end{cor}

Observed that $K_{2}\vee P_{k}=K_{1}\vee(K_{1}\vee P_{k})$. By applying Corollary \ref{cor2-4} twice, we can obtain the following proposition.

\begin{prop}\label{prop2-5}
For all $n\geq k+2$, $sat(n,K_{2}\vee P_{k})\leq 2n-3+sat(n-2,P_{k})$.
\end{prop}

A natural question is whether the equality in Proposition \ref{prop2-5} holds. In this paper, we study this question and obtain Theorem \ref{thm1-1}.

\section{Further results on $\boldsymbol{K_1\vee P_k}$-saturated graphs}
In this section, we present further results on $K_1\vee P_k$-saturated graphs. 
The following two propositions are useful.

\begin{prop}\label{pro3-1*}
Let $G$ be a $K_1\vee H$-saturated graph of order $n$. If $\Delta(G)=n-1$, then $|E(G)|\geq n-1+sat(n-1,H)$, the equality holds if and only if
$G\cong K_1\vee F$, where $F$
is a minimal $H$-saturated graph of order $n-1$.
\end{prop}

\begin{proof}
If $\Delta(G)=n-1$, then $G$ has a conical vertex, say $u$. It follows that $G[V(G)\backslash\{u\}]$ is $H$-saturated by Lemma \ref{lem2-3}, and thus $|E(G[V(G)\backslash\{u\}])|\geq sat(n-1,H)$. Therefore, $|E(G)|=n-1+|E(G[V(G)\backslash\{u\}])|\geq n-1+sat(n-1,H)$, the equality holds if and only if $|E(G[V(G)\backslash\{u\}])|=sat(n-1,H)$, which implies $G[V(G)\backslash\{u\}]$ is a minimal $H$-saturated graph, that is, $G\cong K_1\vee F$, where $F$
is a minimal $H$-saturated graph of order $n-1$.
\end{proof}

\begin{prop}[\rm\!\!\cite{Hu2025}]\label{pro3-2*}
Suppose $H$ is a graph without isolated vertices. If $G$ is a $K_1\vee H$-saturated graph of order $n$, then $diam(G)=2$.
\end{prop}

About a graph with diameter $2$, we have an observation as follows.

\begin{obs}\label{obs1}
Let $G$ be a graph with $diam(G)=2$, $v\in V(G)$ and $N_{2}(v)$ denote the set of vertices with distance $2$ from $v$. Then $N(u)\subseteq N(v)\cup N_{2}(v)$ for any $u\in N_{2}(v)$ and $V(G)=N[v]\cup N_{2}(v)$.
\end{obs}

\begin{lem}\label{lem3-4}
Let $G$ be a $3$-regular graph of order $n$. Then either $G\cong H_4$, or $G$ is not $K_{1}\vee P_{4}$-saturated, where $H_4$ is shown in Figure $1$.
\end{lem}

\begin{proof}
See Appendix A.1.
\end{proof}

Motivated by Theorem \ref{thm2-2}, we obtain the following result, which improves the bounds on $n$ in Theorem \ref{thm2-2}, and characterize the minimal $K_{1}\vee P_{k}$-saturated graphs for $3\leq k\leq 5$.

\begin{thm}\label{thm3-3}
Let $k\geq 3$, $n\geq a_k+1$ and $a_{k}$ defined as $\mathrm{(\ref{eq1.1})}$.
Then
\begin{equation}\label{eq3.1}
sat(n, K_{1}\vee P_{k})=n-1+sat(n-1,P_{k})=
\begin{cases}
\lfloor \frac{3n-3}{2} \rfloor,  & \text{if $k=3$}; \\
\frac{3n-3}{2}, & \text{if $k=4$, $n$ is odd}; \\
\frac{3n}{2}, & \text{if $k=4$, $n$ is even}; \\
\lceil \frac{11n-15}{6} \rceil, & \text{if $k=5$}; \\
2n-2-\lfloor \frac{n-1}{a_{k}} \rfloor, & \text{if $k\geq 6$}.
\end{cases}
\end{equation}
If $G$ is a $K_{1}\vee P_{k}$-saturated graph of order $n$, then $G$ is a minimal $K_{1}\vee P_{k}$-saturated graph if and only if $G\cong K_{1}\vee F$ for $n\geq9$, and $G\in\mathbb{G}$ for $n\leq8$, where $F$ is a minimal $P_{k}$-saturated graph of order $n-1$, $H_i$ $(1\leq i\leq8)$ are shown in Figure $1$, and
\vspace{-0.3cm}
\[\mathbb{G}=
\begin{cases}
\{K_{1}\vee F,H_1,H_2,H_3\},  & \text{if $k=3$}; \\
\{K_{1}\vee F,H_4,H_5,H_6\}, & \text{if $k=4$}; \\
\{K_{1}\vee F,H_7,H_8\}, & \text{if $k=5$}.
\end{cases}\]
\end{thm}

\begin{proof}
See Appendix A.2.
\end{proof}


\section{The proof of Theorem \ref{thm1-1}}
In this section, we prove Theorem \ref{thm1-1}.
Before the proof, we introduce some notation that will be used later.
For convenience, we refer to $K_{2}$  and $P_{k}$ as the center and path part of $K_{2}\vee P_{k}$, respectively, and use $C(K_{2}\vee P_{k})$, $PP(K_{2}\vee P_{k})$ to denote the vertices in center, path part of $K_{2}\vee P_{k}$ correspondingly.

By Proposition \ref{pro3-2*}, the diameter of a $K_{2}\vee P_{k}$-saturated graph is $2$. Next, we present a stronger result.

\begin{prop}\label{pro4-1}
Let $G$ be a $K_{2}\vee P_{k}$-saturated graph with $k\geq3$, $u,v\in V(G)$ and $uv\notin E(G)$. Then there exist at least two paths with length $2$ between $u$ and $v$ in $G$, and thus $diam(G)=2$.
\end{prop}

\begin{proof}
Since $G$ is a $K_{2}\vee P_{k}$-saturated graph, $G+uv$ contains a copy of $K_{2}\vee P_{k}$, say $Q$.

If $u,v\in C(Q)$, then $uxv$ forms a path with length $2$ between $u$ and $v$, where $x\in PP(Q)$, and thus there are at least $k$ paths with length $2$ between $u$ and $v$ in $G$.

If $u\in C(Q),v\in PP(Q)$, then $uwv$ and $uxv$ are two paths with length $2$ between $u$ and $v$ in $G$, where $w\in C(Q)$, $x\in PP(Q)\cap N(v)$.

If $u,v\in PP(Q)$, then $uw_1v$ and $uw_2v$ are two paths with length $2$ between $u$ and $v$ in $G$, where $w_1,w_2\in C(Q)$.

Therefore, there exist at least two paths with length $2$ between $u$ and $v$ in $G$ for any $uv\notin E(G)$, and thus $diam(G)=2$.
\end{proof}

By the proof of Proposition \ref{pro4-1}, the following result is obtained immediately.

\begin{cor}\label{cor4-2}
Let $G$ be a $K_{2}\vee P_{k}$-saturated graph with $k\geq3$, $u,v\in V(G)$, $uv\notin E(G)$, and $Q$ be a copy of $K_2\vee P_{k}$ in $G+uv$. Then $C(Q)\subseteq \{u,v\}\cup (N(u)\cap N(v))$.
\end{cor}

\noindent\textbf{\textit{Proof of Theorem \ref{thm1-1}}.}
Let $F$ be a minimal $P_{k}$-saturated graph of order $n-2$ and $G\in\mathcal{G}$. Then $K_2\vee F$ and $G$ are $K_{2}\vee P_{k}$-saturated graphs for $k\geq3$ by Lemma \ref{lem2-3} and direct checking.

Let $G$ be a $K_{2}\vee P_{k}$-saturated graph with $k\geq3$ of order $n$.
In order to show (\ref{eq1.2}), by Theorem \ref{thm2-1}, we only need to show
\begin{equation}\label{eq4.3}
|E(G)|\geq 2n-3+sat(n-2,P_{k}).
\end{equation}

Select a vertex $v\in V(G)$ with $d_{G}(v)=\delta(G)$ such that $|E(G[N(v)])|$ is minimized. Let $N(v)=\{v_{1}, v_{2},\dots,v_{\delta(G)}\}$ and $U=V(G)\backslash N[v]$.
For any $u\in U$, we know $uv\notin E(G)$, thus $|N(u)\cap N(v)|\geq 2$ by Proposition \ref{pro4-1}.
Let
\begin{itemize}[leftmargin=*]
    \item $U_{i} = \left\{ u \in U \mid |N(u) \cap N(v)| = i \right\}$ for $2 \leq i \leq \delta(G)$;
    \item $X_{t_1t_2\cdots t_i}= \left\{x\mid x\in U_i\cap N(v_{t_1})\cap N(v_{t_2})\cap\cdots\cap N(v_{t_i})\right\}$, where $\{t_1,t_2,\dots, t_i\}\subseteq\{1,2,\dots,\delta(G)\}$.
\end{itemize}
Then $U_{i}=\bigcup_{\{t_1,t_2,\dots,t_i\}\subseteq\{1,2,\dots,\delta(G)\}}X_{t_1t_2\cdots t_i}$, $\bigl| E\bigl[N[v],U\bigr] \bigr|=\sum_{i=2}^{\delta(G)}i|U_i|$, and

\begin{align}
2|E(G)|
&=\sum_{u\in N[v]}d_{G}(u)+\sum_{u\in U}d_{G}(u)
\nonumber \\
&=\left(2\bigl| E(G\bigl[N[v]\bigr]) \bigr|+\bigl| E\bigl[N[v],U\bigr] \bigr|\right)+\left(\bigl| E\bigl[N[v],U\bigr] \bigr|+2\bigl|E(G[U])\bigr|\right)
\nonumber \\
&=2\bigl| E(G\bigl[N[v]\bigr]) \bigr|+2\sum_{i=2}^{\delta(G)}i|U_i|+2\bigl|E(G[U])\bigr|,
\label{eq4.1}
\\[6pt]
|E(G)|
&= \bigl| E(G\bigl[N[v]\bigr]) \bigr| + \sum_{i=2}^{\delta(G)}i|U_i| + \frac{1}{2} \sum_{u \in U} d_{G[U]}(u).
\label{eq4.2}
\end{align}

Since (\ref{eq4.1}) and (\ref{eq4.2}) will be used many times in the proof, we provide some examples to illustrate the use of both equations. Let $\delta_{U}=\min\{d_{G}(u)\mid u\in U\}$. By (\ref{eq4.1}) and (\ref{eq4.2}), we can obtain
\begin{itemize}[leftmargin=*]
    \item $2|E(G)|\geq 2\left(\delta(G)+|E(G[N(v)])|\right)+\sum_{i=2}^{\delta(G)}i|U_i|+\delta_{U}(n-\delta(G)-1)$;
\end{itemize}
\begin{itemize}[leftmargin=*]
\item
$\begin{aligned}[t]
|E(G)|
&\geq \delta(G) + |E(G[N(v)])| + \sum_{i=2}^{\delta(G)} i|U_i| + \frac{1}{2}\sum_{i=2}^{\delta(G)}(\delta_{U_i}-i)|U_i|\\
&\geq \delta(G) + |E(G[N(v)])| + \sum_{i=2}^{\delta(G)} i|U_i| + \frac{1}{2}\sum_{i=2}^{\delta(G)}(\delta(G)-i)|U_i|;
\end{aligned}$
\end{itemize}
\begin{itemize}[leftmargin=*]
    \item For $\{w_1,\ldots,w_t\}\subseteq U$, if $\sum_{i=1}^{t}d_{G}(w_i)>t\delta_{U}$, then $\sum_{u\in U}d_{G}(u)\geq \sum_{i=1}^{t}d_{G}(w_i)+\delta_{U}(n-\delta(G)-1-t)$.
\end{itemize}
%

\noindent\textbf{Claim 1.}
Let $2\leq i\leq \delta(G)$ and $\{t_1,t_2,\dots,t_i\}\subseteq\{1,2,\dots,\delta(G)\}$.
If $\{v_{t_1},v_{t_2},\dots,v_{t_i}\}$ is an independent set of $G$, then $X_{t_1t_2\cdots t_i}=\emptyset$.

\noindent\textbf{\textit{Proof of Claim 1}.} Suppose that $x\in X_{t_1t_2\cdots t_i}\neq\emptyset$. Since $vx\notin E(G)$, $G+vx$ contains a copy of $K_{2}\vee P_{k}$, say $Q$, and $C(Q)\subset \{v,x,v_{t_1},v_{t_2},\ldots,v_{t_i}\}$ by Corollary \ref{cor4-2}. Clearly, $C(Q)\nsubseteq\{v_{t_1},v_{t_2},\ldots,v_{t_i}\}$.
We divide into the following cases.

$\mathbf{Case~1:}$ $C(Q)=\{v,x\}$.

Then $PP(Q)\subseteq N(v)\cap N(x)=\{v_{t_1},v_{t_2},\ldots,v_{t_i}\}$. However, $\{v_{t_1},v_{t_2},\dots,v_{t_i}\}$ is an independent set of $G$, a contradiction.

$\mathbf{Case~2:}$ $C(Q)=\{v,v_{t_j}\}$, where $1\leq j\leq i$.

Then $x\in PP(Q)$ and $PP(Q)\subseteq\{x\}\cup (N(v)\cap N(v_{t_j}))$. On the other hand, $N(v)\cap N(v_{t_j})\subseteq N(v)\backslash\{v_{t_1},v_{t_2},\dots,v_{t_i}\}$ since $\{v_{t_1},v_{t_2},\dots,v_{t_i}\}$ is an independent set of $G$, then $N(x)\cap (N(v)\cap N(v_{t_j}))=\emptyset$, contradicting the fact that $G[\{x\}\cup (N(v)\cap N(v_{t_j}))]$ contains a copy of $P_k$.

$\mathbf{Case~3:}$ $C(Q)=\{x,v_{t_j}\}$, where $1\leq j\leq i$.

Then $PP(Q)\subseteq\{v\}\cup (N(x)\cap N(v_{t_j}))$. But $N(v)\cap N(x)\cap N(v_{t_j})=\emptyset$ since $\{v_{t_1},v_{t_2},\dots,v_{t_i}\}$ is an independent set of $G$, a contradiction.

Therefore, $X_{t_1t_2\cdots t_i}=\emptyset$. This completes the proof.
{\hfill $\square$ \par}

\noindent\textbf{Claim 2.}
Let $k\geq3$, $d_{G}(v)=\delta(G)\leq 5$, $x,y\in N(v)$, $xy\notin E(G)$, $N(x)\cap N(y)\cap N(v)= \emptyset$ and $Q$ be a copy of $K_2\vee P_{k}$ in $G+xy$. Then $v\notin V(Q)$ and $|N_{Q}(x)\cap N_{Q}(y)\cap U|\geq 2$.

\noindent\textbf{\textit{Proof of Claim 2}.}
We have $C(Q)\subseteq \{x,y\}\cup (N(x)\cap N(y))$ by Corollary \ref{cor4-2}.
Firstly, we show $v\notin C(Q)$. If not, then $C(Q)=\{v,x\}$ or $\{v,y\}$. When $C(Q)=\{v,x\}$, then $y\in PP(Q)$. However, $N(x)\cap N(y)\cap N(v)= \emptyset$, this implies $y\notin PP(Q)$, a contradiction. Similarly, $C(Q)\neq\{v,y\}$. Clearly, $v\notin PP(Q)$ by $N(x)\cap N(y)\cap N(v)=\emptyset$. Thus $v\notin V(Q)$, and $C(Q)\subseteq \{x,y\}\cup (N(x)\cap N(y)\cap U)$.

If $C(Q)=\{x,y\}$, then $PP(Q)\subseteq U$, and $|N_{Q}(x)\cap N_{Q}(y)\cap U|=|PP(Q)|=k>2$.

If $|C(Q)\cap U|=1$, without loss of generality, let $C(Q)=\{x,w\}$ with $w\in U$.
Then $y\in PP(Q)$ and $PP(Q)\subset \{y\}\cup(N(x)\cap N(w))$. We take $z\in N(y)\cap PP(Q)$. Then $z\in U$ by $N(x)\cap N(y)\cap N(v)=\emptyset$, and thus $\{w,z\}\subseteq N_{Q}(x)\cap N_{Q}(y)\cap U$, which implies $|N_{Q}(x)\cap N_{Q}(y)\cap U|\geq 2$.

If $|C(Q)\cap U|=2$, then $\{x,y\}\subset PP(Q)$, $C(Q)\subseteq N_{Q}(x)\cap N_{Q}(y)\cap U$, and thus $|N_{Q}(x)\cap N_{Q}(y)\cap U|\geq 2$.
{\hfill $\square$ \par}

Now we show (\ref{eq4.3}) holds based on the value of $\Delta(G)$.

If $\Delta(G)=n-1$, then $|E(G)|\geq n-1+sat(n-1,K_1\vee P_k)=2n-3+sat(n-2,P_k)$ by Proposition \ref{pro3-1*}, $K_2\vee P_k=K_1\vee(K_1\vee P_k)$ and Theorem \ref{thm3-3}, (\ref{eq4.3}) holds, with the equality if and only if $G\cong K_1\vee F'$, where $F'$ is a minimal $K_{1}\vee P_{k}$-saturated graph of order $n-1$.
Furthermore, by Theorem \ref{thm3-3}, the equality in (\ref{eq4.3}) holds if and only if $G\cong K_{2}\vee F$ for $n\geq10$, and $G\cong K_1\vee G^*$ for $n\leq9$, where $F$ is a minimal $P_{k}$-saturated graph of order $n-2$ and $G^*\in\mathbb{G}$.


In the rest of the proof, we assume $\Delta(G)\leq n-2$. Since $diam(G)=2$, we have $\delta(G)\geq 2$.
Clearly, $|E(G[N(v)])|\neq 0$. Otherwise, $U=\emptyset$ by Claim $1$, which implies $G\cong K_{1,n-1}$, a contradiction with $\Delta(G)\leq n-2$.

If $\delta(G)=2$, then $v_1v_2\in E(G)$ by $|E(G[N(v)])|\neq 0$, and thus $\Delta(G)=d_{G}(v_1)=d_{G}(v_2)=n-1$ by Proposition \ref{pro4-1}, a contradiction. Thus $3\leq \delta(G)\leq\Delta(G)\leq n-2$.

\noindent\textbf{Claim 3.}
Let $k\geq3$, $3\leq\delta(G)=d_{G}(v)\leq\Delta(G)\leq n-2$ and $|E(G[N(v)])|=\delta(G)-1$. Then $G[N(v)]\not\cong K_{1,\delta(G)-1}$.

\noindent\textbf{\textit{Proof of Claim 3}.}
If $G[N(v)]\cong K_{1,\delta(G)-1}$, without loss of generality, we assume $v_1v_2,v_1v_3,\ldots,\\v_1v_\delta(G)\in E(G)$. By Claim $1$, we have
$U_{i}=\bigcup_{\{t_1,t_2,\dots,t_{i-1}\}\subseteq\{2,\dots,\delta(G)\}}X_{1t_1\cdots t_{i-1}}$ for $i\geq2$, and thus $uv_1\in E(G)$ for any $u\in U$, which implies $d_{G}(v_1)=n-1$, a contradiction.
{\hfill $\square$ \par}

\noindent\textbf{Claim 4.}
Let $k\geq3$, $3=\delta(G)=d_{G}(v)\leq\Delta(G)\leq n-2$. Then $|E(G[N(v)])|\neq2$.

\noindent\textbf{\textit{Proof of Claim 4}.}
If $|E(G[N(v)])|=2$, then $G[N(v)]\cong K_{1,2}$ by $d_{G}(v)=3$, a contradiction with Claim $3$.
{\hfill $\square$ \par}

In the rest, we divide the remaining proof into four parts based on the value of $k$.

\subsection{$\boldsymbol{k\geq 6}$ with $\boldsymbol{3\leq \delta(G)\leq\Delta(G)\leq n-2}$}
In this part, we only show $|E(G)|\geq 3n-5$, which implies that (\ref{eq4.3}) holds strictly.

$\mathbf{Case~1:}$ $\delta(G)\geq 6$.

Then $|E(G)|\geq \frac{n\delta(G)}{2}\geq3n$, (\ref{eq4.3}) holds.

$\mathbf{Case~2:}$ $\delta(G)=3$.

Clearly, $|E(G[N(v)])|=1$ or $3$ by Claim $4$.

$\mathbf{Subcase~2.1:}$ $|E(G[N(v)])|=1$.

\noindent\textbf{Claim 5.}
Let $k\geq4$, $\delta(G)=3\leq\Delta(G)\leq n-2$ and $|E(G[N(v)])|=1$. Then (\ref{eq4.3}) holds.

\noindent\textbf{\textit{Proof of Claim 5}.}
Without loss of generality, let $v_1v_2\in E(G)$. Then $\{v_1,v_3\}$ and $\{v_2,v_3\}$ are independent sets of $G$, and thus $X_{13}=X_{23}=\emptyset$,
$V(G)=N[v]\cup X_{12}\cup X_{123}$, $|X_{123}|\geq 2$ by $d_{G}(v_3)\geq \delta(G)=3$.

For any $y\in X_{123}$, $G+vy$ must contain a copy of $K_{2}\vee P_{k}$, say $Q$. Then $C(Q)\subset\{v,y,v_1,v_2,v_3\}$ by Corollary \ref{cor4-2}. Clearly, $v\notin C(Q)$ by $d_{G+vy}(v)=4<k+1$, then $v\in PP(Q)$. If $v_3\in C(Q)$, then $C(Q)=\{v_3,y\}$ by $v_3v_1,v_3v_2\notin E(G)$, and thus $v\notin PP(Q)$ by $N(v)\cap N(v_3)\cap N(y)=\emptyset$, a contradiction.
Then $C(Q)=\{y,v_i\}$ for $i\in\{1,2\}$ or $C(Q)=\{v_1,v_2\}$. If $C(Q)=\{y,v_i\}$, where $i\in\{1,2\}$, then $d_{Q}(y)=k+1$, $d_{G}(y)\geq k\geq4$ and $G[U]$ contains a copy of $K_1\vee P_{k-2}$. If $C(Q)=\{v_1,v_2\}$, we have $v_3\notin PP(Q)$, $G[U]$ contains a copy of $P_{k-1}$, $d_{Q}(v)=3$, $d_{Q}(y)=4$, which implies $d_{G}(y)\geq (d_{Q}(y)-1)+1=4$. Therefore, we have $d_{G}(y)\geq 4$ for any $y\in X_{123}$.

If $X_{12}=\emptyset$, then $V(G)=N[v]\cup X_{123}$ and $U=X_{123}$.
Since $G[X_{123}]$ contains a copy of $K_1\vee P_{k-2}$ or $P_{k-1}$, we have $|E(G[X_{123}])|\geq |E(P_{k-1})|=k-2$, and by (\ref{eq4.2}),
$|E(G)|\geq |E(G[N[v]])|+3(n-4)+k-2=3n+k-10$.

If $X_{12}\neq \emptyset$, for any $x\in X_{12}$, we have $xv_3\notin E(G)$, then $d_{G}(x)\geq 4$ since there exist at least two paths with length $2$ between $v_3$ and $x$ by Proposition \ref{pro4-1}, and $G+v_{3}x$ contains a copy of $K_{2}\vee P_{k}$, say $Q'$, we have $C(Q')\subset\{v_3,x\}\cup (N(v_3)\cap N(x))$ by Corollary \ref{cor4-2}.

Now we show $v_3\in PP(Q')$. In fact, if $v_3\in C(Q')$, then $C(Q')=\{v_3,x\}$, or $\{v_3,y\}$, where $y\in N(v_3)\cap N(x)\subseteq X_{123}$.
When $C(Q')=\{v_3,x\}$, then $PP(Q')\subseteq X_{123}$, which implies $G[PP(Q')\cup\{v_1,v_2\}]$ and thus $G$ contains a copy of $K_{2}\vee P_{k}$, a contradiction.
When $C(Q')=\{v_3,y\}$, similar to the proof of $C(Q')=\{v_3,x\}$, we can obtain a contradiction.

Therefore, $v_3\in PP(Q')$ and $C(Q')\subset\{x\}\cup (N(v_3)\cap N(x))$. Let $C(Q')=\{w_1,w_2\}$. Then $d_{G}(w_1)+d_{G}(w_2)\geq 2k+1$. By (\ref{eq4.1}), we have
$2|E(G)|\geq 2|E(G[N[v]])|+2|X_{12}|+3|X_{123}|+2k+1+4(n-6)=6n+2k-23+|X_{123}|\geq 6n+2k-21$.

Combining above arguments, we have $|E(G)|>3n-5$ for $k\geq6$, $|E(G)|\geq 3n-5>\lceil \frac{17n-32}{6} \rceil$ for $k=5$, $|E(G)|\geq 3n-6\geq\frac{5n-5}{2}$ for $k=4$, then (\ref{eq4.3}) holds.
{\hfill $\square$ \par}

$\mathbf{Subcase~2.2:}$ $|E(G[N(v)])|=3$.

Then $U=X_{12}\cup X_{13}\cup X_{23}\cup X_{123}$ with $U_2=X_{12}\cup X_{13}\cup X_{23}$ and $U_3=X_{123}$. Thus
$X_{12}\neq\emptyset$, $X_{13}\neq\emptyset$ and $X_{23}\neq\emptyset$ by $\Delta(G)\leq n-2$.

\noindent\textbf{Claim 6.}
Let $k\geq3$, $\delta(G)=3\leq\Delta(G)\leq n-2$ and $|E(G[N(v)])|=3$. If $X_{123}=\emptyset$ or $d_{G}(u)=3$ for any $u\in X_{123}$, then $G[U_2]$ is connected.

\noindent\textbf{\textit{Proof of Claim 6}.}
Since $|E(G[N(v)])|=3$, we have $X_{ij}\neq\emptyset$ for $1\leq i<j\leq3$.
Let $x,y\in U_2$.

If $x\in X_{12}$ and $y\in X_{13}$, there exist at least two paths with length $2$ between $x$ and $y$ in $G$ by Proposition \ref{pro4-1}. In fact, except for the path $xv_1y$, there exists a path between $x$ and $y$ in $G[U_2]$ by $X_{123}=\emptyset$ or $d_{G}(u)=3$ for any $u\in X_{123}$.

If $x,y\in X_{12}$, by above arguments, there exist a path with length $2$ between $x$ and $z$, and a path with length $2$ between $z$ and $y$ in $G[U_2]$, where $z\in X_{13}$, and thus there exists a path connecting $x$ and $y$ in $G[U_2]$.

For other cases, the proofs are similar, and we omit them. Therefore, $G[U_2]$ is connected. In fact, if $X_{123}=\emptyset$, we have $U=U_2$, and $G[U]$ is connected.
{\hfill $\square$ \par}

$\mathbf{Subcase~2.2.1:}$ $X_{123}=\emptyset$.

Then $V(G)=N[v]\cup U$ with $U=U_2=X_{12}\cup X_{13}\cup X_{23}$.
For any $w\in U_2$, $G+vw$ contains a copy of $K_{2}\vee P_{k}$, say $Q_w$, we have $C(Q_w)\subset\{v,w\}\cup(N(v)\cap N(w))$ by Corollary \ref{cor4-2}. Clearly, $v\notin C(Q_w)$ by $d_{G}(v)=3$, and $v\in PP(Q_w)$.

If $w\in C(Q_w)$, then $C(Q_w)\subset\{w\}\cup N(v)$, $v\in PP(Q_w)$ and $N(v)\cap PP(Q_w)\neq \emptyset$, and thus $G[U]$ contains a copy of $K_{1}\vee P_{k-2}$.
It follows that $|E(G[U])|\geq |U|-1+k-3=n+k-8$ by Claim $6$, and thus $|E(G)|\geq3n-4$ by (\ref{eq4.2}), (\ref{eq4.3}) holds.

If $w\notin C(Q_w)$, then $C(Q_w)\subseteq N(v)\cap N(w)\subset \{v_1,v_2,v_3\}$ and $\{v,w\}\subset PP(Q_w)$. Now we show $G[X_{12}]$, $G[X_{13}]$ and $G[X_{23}]$ all contain cycles.

For $w\in X_{12}$, we have $C(Q_w)= \{v_1,v_2\}$ by $wv_3\notin E(G)$ and $\{v,w\}\subset PP(Q_w)$, then there exists a path $P_{k-2}$ (if $v_3\in PP(Q_w)$) or $P_{k-1}$ (if $v_3\notin PP(Q_w)$) with endpoint $w$ in $G[X_{12}]$. Let $G_1$ be a connected component of $G[X_{12}]$ such that $P=ww_1\cdots w_{k-3}$ is a path with length $k-3$ in $G_1$, $l$ be the length of the longest path in $G_1$.
Then $k-3\leq l\leq k-2$. Otherwise, if $l\geq k-1$, then $G[\{v_1,v_2\}\cup V(G_1)]$ and thus $G$ contains a copy of $K_2\vee P_k$, a contradiction.

If $l=k-3$, then $P$ is the longest path in $G_1$. Similar to above arguments for $w$, there exists a path $P^1$ with length $k-3$ and endpoint $w_1$ in $G_1$, then $P\cup P^1$ must contain a cycle. Otherwise, there exists a path with length longer than $k-3$ in $G_1$, a contradiction.

If $l=k-2$, we assume $P'=w'_{1}w'_{2}\cdots w'_{k-1}$ is the longest path in $G_1$. Since there exists a path $P^2$ with length $k-3$ and endpoint $w'_3$ in $G_1$, it follows that $P^2\cup P'$ contains a cycle.
Therefore, $G_1$ and thus $G[X_{12}]$ contains a cycle.

For $w\in X_{13}$ or $X_{23}$, we can show $G[X_{13}]$ and $G[X_{23}]$ also contain a cycle similarly.

Combining above arguments, there are three disjoint cycles in $G[U]$, then $|E(G[U])|\geq |U|-1+3=n-2$ by Claim $6$, and $|E(G)|\geq 3n-4$ by (\ref{eq4.2}), (\ref{eq4.3}) holds.

$\mathbf{Subcase~2.2.2:}$ $X_{123}\neq\emptyset$.

Let $G_1,G_2,\dots,G_t$ be all connected components of $G[U]$. By (\ref{eq4.2}), we have
\begin{equation}\label{eq4.4}
|E(G)|\geq 6+2|U_{2}|+3|U_{3}|+\sum_{i=1}^{t}(|V(G_i)|-1)=6+3(n-4)+|X_{123}|-t.
\end{equation}

For any $w\in X_{12}$, without loss of generality, we take $w\in V(G_1)$. By Proposition \ref{pro4-1}, for any $x\in X_{13}\cup X_{23}$, there exists a path between $w$ and $x$ in $G[U]$, thus $X_{12}\cup X_{13}\cup X_{23}\subseteq V(G_1)$, and if $t\geq2$, $V(G_{j})\subseteq X_{123}$ for any $j\in\{2,\dots,t\}$.

\noindent\textbf{Claim 7.}
Let $z_1,z_2\in X_{12}\cup X_{13}\cup X_{23}$, $z_1z_2\notin E(G)$ and $|N(z_1)\cap N(z_2)\cap N(v)|=1$. Then there exists a path with length $2$ between $z_1$ and $z_2$ in $G_1$.

\noindent\textbf{\textit{Proof of Claim 7}.} Since $z_1z_2\notin E(G)$, there exist at least two paths with length $2$ between $z_1$ and $z_2$ in $G$ by Proposition \ref{pro4-1}, one is $z_1v_iz_2$, where $\{v_i\}=N(z_1)\cap N(z_2)\cap N(v)$, and the other must be in $G_1$ since $X_{12}\cup X_{13}\cup X_{23}\subseteq V(G_1)$.
{\hfill $\square$ \par}

$\mathbf{Subcase~2.2.2.1:}$ $V(G_1)\cap X_{123}=\emptyset$.

Then $t\geq 2$, $X_{123}=\bigcup_{j=2}^{t}V(G_j)$, and $V(G_1)=U_2=X_{12}\cup X_{13}\cup X_{23}$.

If there exists some $G_i$ such that $|V(G_i)|=1$ with $i\in\{2,\dots,t\}$, we take $V(G_i)=\{x\}$. Since $G+vx$ contains a copy of $K_{2}\vee P_{k}$, say $Q$, we have $v,x\notin C(Q)$ by $d_{G}(v)=d_{G}(x)=3$, and $C(Q)\subset\{v_1,v_2,v_3\}$ by Corollary \ref{cor4-2}. Without loss of generality, let $C(Q)=\{v_1,v_2\}$.
Then $PP(Q)\backslash\{x,v,v_3\}$ must be in the same component of $G[X_{123}]$, without loss of generality, we assume $PP(Q)\backslash\{x,v,v_3\}\subseteq V(G_j)$, where $j\in\{2,\dots,t\}\backslash\{i\}$. Then $|V(G_j)|\geq k-3$. On the other hand, $|V(G_l)|=1$ for any $l\in\{2,\dots,t\}\backslash\{j\}$. Otherwise, if $|V(G_l)|\geq 2$ with $l\in \{2,\dots,t\}\backslash\{i,j\}$, then $G[V(G_l)\cup V(G_j)\cup N(v)]$ contains a copy of $K_{2}\vee P_{k}$, say $Q'$ with $C(Q')=\{v_1,v_2\}$ and $PP(Q')\subseteq V(G_l)\cup\{v_3\}\cup V(G_j)$, a contradiction. Then $|X_{123}|=t-2+|V(G_j)|\geq t+k-5$. Therefore, $|E(G)|\geq 6+3(n-4)+k-5\geq 3n-5$ by (\ref{eq4.4}), (\ref{eq4.3}) holds.

If $|V(G_i)|\geq 2$ for any $i\in \{2,\dots,t\}$ and there exists some $G_j$ with $j\in\{2,\dots,t\}$ such that $|V(G_j)|\geq 3$, then $|X_{123}|\geq 2(t-2)+3=2t-1$, and $|E(G)|\geq 6+3(n-4)+t-1\geq 3n-5$ by (\ref{eq4.4}), (\ref{eq4.3}) holds.

If $|V(G_i)|=2$ for any $i\in\{2,\dots,t\}$, then $|X_{123}|=2(t-1)$.
Let $z\in V(G_2)$. Since $G+vz$ contains a copy of $K_{2}\vee P_{k}$, say $Q''$, we have $v,z\notin C(Q'')$ by $d_{G}(v)=3$ and $d_{G}(z)=4$, $C(Q'')\subset\{v_1,v_2,v_3\}$ by Corollary \ref{cor4-2} and $PP(Q'')\subseteq X_{123}\cup(N[v]\backslash C(Q''))$.
Then $k=6$, $t\geq 3$, and thus $|E(G)|\geq 6+3(n-4)+2(t-1)-t\geq 3n-5$ by (\ref{eq4.4}), (\ref{eq4.3}) holds.

$\mathbf{Subcase~2.2.2.2:}$ $V(G_1)\cap X_{123}\neq\emptyset$.

Then $|X_{123}|\geq t$ since $V(G_1)\cap X_{123}\neq\emptyset$, and if $t\geq2$, $V(G_{j})\subseteq X_{123}$ for $2\leq j\leq t$.

If $|V(G_i)\cap X_{123}|\geq 2$ for some $1\leq i\leq t$, then $|X_{123}|\geq t+1$, and thus $|E(G)|\geq 3n-5$ by (\ref{eq4.4}), (\ref{eq4.3}) holds.

If $G_1$ contains a cycle, then $|E(G[U])|\geq |V(G_1)|+\sum_{i=2}^{t}(|V(G_i)|-1)=n-t-3$, and thus $|E(G)|\geq 6+2|U_{2}|+3|U_{3}|+n-t-3\geq 3n-5$ by (\ref{eq4.2}), (\ref{eq4.3}) holds.

If $G_1$ contains no cycle and $|V(G_i)\cap X_{123}|=1$ for any $i\in\{1,2,\dots,t\}$, we take $V(G_1)\cap X_{123}=\{x_1\}$, $V(G_j)=\{x_j\}$ for $2\leq j\leq t$. Then $V(G_1)=X_{12}\cup X_{13}\cup X_{23}\cup\{x_1\}$.

Let $Q$ be the copy of $K_{2}\vee P_{k}$ in $G+vx_1$. Then $v\in PP(Q)$ by $d_{G}(v)=3$, $C(Q)\subset\{x_1,v_1,v_2,v_3\}$ by Corollary \ref{cor4-2}. Moreover, $x_1\notin C(Q)$. Otherwise, if $x_1\in C(Q)$, then $|PP(Q)\cap U_2|\geq k-3\geq3$, and thus $G[\{x_1\}\cup(PP(Q)\cap U_2)]$ contains a cycle in $G_1$ by
$x_1\in C(Q)$, a contradiction. Without loss of generality, let $C(Q)=\{v_1,v_2\}$.

Firstly, we consider $t=1$. If $v_3\notin PP(Q)$, then $G[(V(Q)\setminus\{v\})\cup\{v_3\}]$ and thus $G$ contains a copy of $K_{2}\vee P_{k}$, a contradiction. If $v_3\in PP(Q)$, then $Q-vx_1+v_3x_1$ is a copy of $K_{2}\vee P_{k}$ in $G$, a contradiction.

Next, we consider $t\geq2$. Then there exists some $i\in\{2,\dots,t\}$ such that $x_i\in PP(Q)$ by the proof of $t=1$ and $\{x_2,\ldots,x_t\}$ is an independent set of $G$.
Therefore, $\{v,v_3,x_1,x_i\}\subset PP(Q)$, $|PP(Q)\cap G_1|=k-3$, and there exists a path with length $k-4$ and endpoint $x_1$ in $G_1$, say, $P=x_1w_1w_2\cdots w_{k-4}$, where $w_{i}\in X_{12}$ for any $1\leq i\leq k-4$.

Let $y\in X_{13}$. Then $G+x_2y$ contains a copy of $K_2\vee P_k$, say $Q'$. Clearly, $x_2\notin C(Q')$ by $d_{G}(x_2)=3$. If $y\in C(Q')$, we assume $C(Q')=\{y,v_i\}$ for some $i\in\{1,3\}$ by Corollary \ref{cor4-2} and $yv_2\notin E(G)$, then there exists a copy of $K_1\vee P_{k-2}$ in $G_1$, thus $G_1$ contains a cycle by $k-2>2$, a contradiction. Therefore, $C(Q')=\{v_1,v_3\}$ by Corollary \ref{cor4-2}, $x_2,y\in PP(Q')$.

If $d_{Q'}(y)=3$, then $yx_2,x_2v_2,v_2x_1\in E(Q')$, and there exists a path $P'$ with length $k-4$ and endpoint $x_1$ in $G_1$, say, $P'=x_1y_1y_2\cdots y_{k-4}$, where $y_i\in X_{13}$ for any $i\in\{1,2,\dots,k-4\}$, and the path part of $Q'$ is $yx_2v_2x_1+P'$. Since $G_1$ contains no cycle, we have $y_{k-4}w_{k-4}\notin E(G)$, and $P+P'$ is the unique path with length $2k-8$ ($>2$) between $y_{k-4}$ and $w_{k-4}$ in $G_1$, a contradiction with Claim $7$.

If $d_{Q'}(y)=4$, we take $y'\in V(G_1)\cap PP(Q')\cap N(y)$. Since $y'v_1,y'v_3\in E(G)$, we have $y'=x_1$ or $y'\in X_{13}$. When $y'=x_1$, then $yw_{k-4}\notin E(G)$ and $yx_1+P$ is the unique path with length $k-3$ ($>2$) between $y$ and $w_{k-4}$ in $G_1$ since $G_1$ contains no cycle, a contradiction with Claim $7$. When $y'\in X_{13}$, then $|E_{G_1}[\{y,y'\},V(P)]|\leq1$ since $G_1$ contains no cycle. Thus there exist $w_i,w_j\in V(P)\backslash\{x_1\}$ ($i=j$ is possible) such that $yw_i$, $y'w_j\notin E(G)$, similar to above arguments, we can obtain a contradiction with Claim $7$.

$\mathbf{Case~3:}$ $\delta(G)=4$.

Then $1\leq |E(G[N(v)])|\leq 6$ by Claim $1$.

\noindent\textbf{Claim 8.}
Let $k\geq4$, $\delta(G)=4\leq\Delta(G)\leq n-2$ and $|E(G[N(v)])|=1$. Then (\ref{eq4.3}) holds.

\noindent\textbf{\textit{Proof of Claim 8}.}
Without loss of generality, let $v_1v_2\in E(G)$. Then $V(G)=N[v]\cup X_{12}\cup X_{123}\cup X_{124}\cup X_{1234}$ by Claim $1$.
Since $d_{G}(v_i)\geq \delta(G)=4$ for $i\in\{3,4\}$, we have $|U_3|+2|U_4|=|E[\{v_3,v_4\},U]|=d_{G}(v_3)+d_{G}(v_4)-2\geq 6$ by $U_2=X_{12}$, $U_3=X_{123}\cup X_{124}$ and $U_4=X_{1234}$, which implies $|U|=n-5\geq|U_3|+|U_4|\geq3$, say, $n\geq8$.

If $X_{12}\neq\emptyset$, we take $x\in X_{12}$, $G+v_3x$ contains a copy of $K_{2}\vee P_{k}$, say $Q$, we have $v_3\notin C(Q)$. Otherwise, if $v_3\in C(Q)$, then $PP(Q)\subseteq X_{12}\cup X_{123}\cup X_{1234}$, thus $G[\{v_1,v_2\}\cup PP(Q)]$ and $G$ contain a copy of $K_{2}\vee P_{k}$, a contradiction.
Thus $C(Q)\subset U$ and $\sum_{u\in C(Q)}d_{G}(u)\geq 2k+1$.
Therefore, by (\ref{eq4.1}) and $|U_3|+2|U_4|\geq6$, we have $2|E(G)|\geq (10+2|U_2|+3|U_3|+4|U_4|)+(2k+1+4(n-7))\geq 6n+2k-21$.

If $X_{12}=\emptyset$, then $U=X_{123}\cup X_{124}\cup X_{1234}$.
For any $x\in X_{123}\neq \emptyset$, then $d_{G}(x)\geq 5$ since there exist at least two paths with length $2$ between $x$ and $v_4$ in $G$ by Proposition \ref{pro4-1}. Similarly, $d_{G}(y)\geq 5$ for any $y\in X_{124}\neq \emptyset$. For any $z\in X_{1234}\neq \emptyset$, $G+vz$ contains a copy of $K_{2}\vee P_{k}$, say $Q'$, then $v\notin C(Q')$ since $G[N[v]\cup\{z\}]$ contains no copy of $K_{2}\vee P_{k}$ for $k\geq4$, thus $C(Q')=\{v_1,v_2\}$ or $\{z,v_i\}$ for some $i\in\{1,2\}$ by Corollary \ref{cor4-2} and $|E(G[N(v)])|=1$, and $|V(Q')\cap U|=k-1$. It follows that $d_{G}(z)\geq 5$ and $|U|=n-5\geq k-1$. Therefore, $d_{G}(u)\geq 5$ for any $u\in U$ and $n\geq k+4$. On the other hand, since $G+v_3v_4$ contains a copy of $K_{2}\vee P_{k}$, say $Q''$, we have $|N_{Q''}(v_3)\cap N_{Q''}(v_4)\cap U|\geq 2$ by Claim $2$, which implies $|U_4|\geq 2$.  By (\ref{eq4.2}), we have $|E(G)|\geq 5+3|U_3|+4|U_4|+(|U_3|+0.5|U_4|)\geq4n-14\geq 3n+k-10$.

Combining above arguments, we have $|E(G)|>3n-5$ for $k\geq6$, $|E(G)|\geq 3n-5>\lceil \frac{17n-32}{6} \rceil$ for $k=5$, $|E(G)|\geq 3n-6>\frac{5n-5}{2}$ for $k=4$, then (\ref{eq4.3}) holds.
{\hfill $\square$ \par}

$\mathbf{Subcase~3.1:}$
$|E(G[N(v)])|=1$.

Clearly, (\ref{eq4.3}) holds by Claim $8$.

$\mathbf{Subcase~3.2:}$ $|E(G[N(v)])|=2$.

Then $G[N(v)]\cong P_3\cup K_1$ or $G[N(v)]\cong 2K_2$. For any $x\in U_4=X_{1234}$, $d_{G}(x)\geq 5$, then $d_{G[U]}(x)\geq1$ since $G+vx$ must contain a copy of $K_{2}\vee P_{k}$ whether $x$ is in the center or path part. By $\delta(G)=4$ and (\ref{eq4.2}), we have
\begin{align}
|E(G)|
&\geq 6+2|U_2|+3|U_3|+4|U_4|+(|U_2|+0.5|U_3|+0.5|U_4|)
\nonumber \\
&=3n-9+0.5|U_3|+1.5|U_4|.
\label{eq4.5}
\qedhere
\end{align}

If there exists a vertex $u\in U$ with $d_{G}(u)\geq k+1$, by (\ref{eq4.1}), then
\begin{align}
2|E(G)|
&\geq (12+2|U_2|+3|U_3|+4|U_4|)+(k+1+4(n-6))
\nonumber \\
&=6n+|U_3|+2|U_4|+k-21.
\label{eq4.6}
\qedhere
\end{align}

\noindent\textbf{Claim 9.}
Let $k\geq4$, $\delta(G)=4\leq\Delta(G)\leq n-2$ and $G[N(v)]\cong P_3\cup K_1$. Then (\ref{eq4.3}) holds.

\noindent\textbf{\textit{Proof of Claim 9}.}
Without loss of generality, let $v_1v_2,v_2v_3\in E(G)$. Then $X_{13}=X_{14}=X_{24}=X_{34}=X_{134}=\emptyset$ by Claim $1$, and $V(G)=N[v]\cup X_{12}\cup X_{23}\cup X_{123}\cup X_{124}\cup X_{234}\cup X_{1234}$ with $U_2=X_{12}\cup X_{23}$, $U_3=X_{123}\cup X_{124}\cup X_{234}$ and $U_4=X_{1234}$.

Let $Q$, $Q'$ be the copy of $K_{2}\vee P_{k}$ in $G+v_3v_4$ and $G+v_1v_4$, respectively. Then $|X_{234}\cup X_{1234}|=|N_{Q}(v_3)\cap N_{Q}(v_4)\cap U|\geq 2$ and $|X_{124}\cup X_{1234}|=|N_{Q'}(v_1)\cap N_{Q'}(v_4)\cap U|\geq 2$ by Claim $2$. Moreover, by Claim $2$, Corollary \ref{cor4-2} and $X_{14}=X_{34}=\emptyset$, we have $v\notin V(Q)$, $C(Q)\subseteq\{v_3,v_4\}\cup X_{234}\cup X_{1234}$, $v\notin V(Q')$ and $C(Q')\subset\{v_1,v_4\}\cup X_{124}\cup X_{1234}$.

In fact, $|X_{124}\cup X_{234}\cup X_{1234}|\geq 3$ by $d_{G}(v_4)\geq\delta(G)=4$, which implies $|U|=n-5\geq|U_3|+|U_4|\geq3$, say, $n\geq8$.

If $|U_4|\geq 3$, then $|E(G)|\geq 3n-4$ by (\ref{eq4.5}).

If $|U_4|=2$, then $|U_3|\geq1$ by $|X_{124}\cup X_{234}\cup X_{1234}|\geq 3$, thus $|E(G)|\geq 3n-5$ by (\ref{eq4.5}).

If $|U_4|=1$, then $|U_3|\geq 2$ and $|X_{234}|\geq1$ by $|X_{124}\cup X_{234}\cup X_{1234}|\geq 3$ and $|X_{234}\cup X_{1234}|\geq 2$, and thus $|U_3|+2|U_4|\geq 4$.
Let $U_4=\{x\}$. Then $d_{G}(x)\geq 5$ by considering $G+vx$.

When $C(Q')=\{v_1,v_4\}$, then $PP(Q')\subseteq X_{124}\cup X_{1234}$, and $|X_{124}|\geq k-1$ by $|X_{1234}|=1$, thus $|U_3|\geq|X_{124}|+|X_{234}|\geq k$, and $|E(G)|\geq 3n-5$ by (\ref{eq4.5}).

When $C(Q')\subseteq X_{124}\cup X_{1234}$, then $\sum_{u\in C(Q')}d_{G}(u)\geq2k
+2$, thus $2|E(G)|\geq (12+2|U_2|+3|U_3|+4|U_4|)+(2k+2+4(n-7))\geq 6n+2k-20$ by (\ref{eq4.1}).

When $|C(Q')\cap(X_{124}\cup X_{1234})|=1$, we have $v_3\notin PP(Q')$ and $|PP(Q')\cap U|\geq k-2$. Let $u\in C(Q')\cap(X_{124}\cup X_{1234})$. If $u\in X_{124}$, then $d_{G}(u)\geq k+1$, and thus $2|E(G)|\geq (12+2|U_2|+3|U_3|+4|U_4|)+(d_{G}(u)+d_{G}(x)+4(n-7))\geq 6n+k-16$ by (\ref{eq4.1}).
If $u\in X_{1234}$, then $d_{G}(u)\geq k+2$, and thus $2|E(G)|\geq (12+2|U_2|+3|U_3|+4|U_4|)+(d_{G}(u)+4(n-6))\geq 6n+k-16$ by (\ref{eq4.1}).

If $|U_4|=0$, then $|X_{234}|\geq 2$ by $|X_{234}\cup X_{1234}|\geq 2$. By considering $Q'$ in $G+v_1v_4$, we have $|X_{124}|\geq 2$ by $X_{1234}=U_4=\emptyset$. Similar to the proof of Claim $2$, by considering the center of the copy of $K_2\vee P_k$ in $G+v_1v_3$, we have $|X_{123}|\geq 1$. Thus $|U_3|=|X_{123}|+|X_{124}|+|X_{234}|\geq 5$.
When $C(Q)=\{v_3,v_4\}$, then $|X_{234}|\geq k$, thus $|U_3|\geq k+3$, and $|E(G)|\geq 3n-5$ by (\ref{eq4.5}).
When $C(Q)\neq\{v_3,v_4\}$, we take $w\in C(Q)\cap X_{234}\neq\emptyset$, then $d_{G}(w)\geq k+1$, and thus $2|E(G)|\geq 6n+k-16$ by (\ref{eq4.6}) and $|U_3|\geq 5$.

Combining above arguments, we have $|E(G)|\geq3n-5$ for $k\geq6$, $|E(G)|\geq 3n-5>\lceil \frac{17n-32}{6} \rceil$ for $k=5$, $|E(G)|\geq 3n-6>\frac{5n-5}{2}$ for $k=4$, then (\ref{eq4.3}) holds.
{\hfill $\square$ \par}

\noindent\textbf{Claim 10.}
Let $k\geq4$, $\delta(G)=4\leq\Delta(G)\leq n-2$ and $G[N(v)]\cong 2K_2$. Then (\ref{eq4.3}) holds.

\noindent\textbf{\textit{Proof of Claim 10}.}
Without loss of generality, let $v_1v_2, v_3v_4\in E(G)$. Then
$X_{13}=X_{14}=X_{23}=X_{24}=\emptyset$ by Claim $1$, and $V(G)=N[v]\cup X_{12}\cup X_{34}\cup X_{123}\cup X_{124}\cup X_{134}\cup X_{234}\cup X_{1234}$ with $U_2=X_{12}\cup X_{34}$, $U_3=X_{123}\cup X_{124}\cup X_{134}\cup X_{234}$ and $U_4=X_{1234}$. Since $G+v_1v_3$ and $G+v_2v_4$ must contain a copy of $K_{2}\vee P_{k}$, we have $|X_{123}\cup X_{134}\cup X_{1234}|\geq2$ and $|X_{124}\cup X_{234}\cup X_{1234}|\geq2$ by Claim $2$, which implies $|U_3|+2|U_4|\geq4$.

If $|U_4|\geq 3$, then $|E(G)|\geq 3n-5$ by (\ref{eq4.5}).

If $|U_4|\leq 2$, we consider the copy of $K_{2}\vee P_{k}$ in $G+v_iv_j$, say $Q_{ij}$, where $i\in\{1,2\}$ and $j\in\{3,4\}$. Obviously, $v\notin V(Q_{ij})$ by Claim $2$, and $C(Q)\subseteq\{v_i,v_j\}\cup U$ by Corollary \ref{cor4-2}.

If there exist $i\in\{1,2\}$ and $j\in\{3,4\}$ such that $C(Q_{ij})=\{v_i,v_j\}$, without loss of generality, let $i=1$ and $j=3$. Then $PP(Q_{13})\subseteq X_{123}\cup X_{134}\cup X_{1234}$, thus $|X_{123}\cup X_{134}\cup X_{1234}|\geq k$, which implies $|U|=n-5\geq k$, say, $n\geq9$. By $|X_{124}\cup X_{234}\cup X_{1234}|\geq 2$, we have $|U_3|\geq k-2$ for $|U_4|=2$, $|U_3|\geq k$ for $|U_4|=1$ and $|U_3|\geq k+2$ for $|U_4|=0$, then $|E(G)|\geq 3n-5$ for $k\geq5$ and $|E(G)|\geq 3n-6$ for $k=4$ by (\ref{eq4.5}).

If $C(Q_{ij})\neq\{v_i,v_j\}$ for any $i\in\{1,2\}$ and $j\in\{3,4\}$, then $1\leq|C(Q_{ij})\cap U|\leq2$. If there exist $i,j$ such that $|C(Q_{ij})\cap U|=2$, then $\sum_{u\in C(Q_{ij})}d_{G}(u)\geq2k+2$, thus $2|E(G)|\geq (12+2|U_2|+3|U_3|+4|U_4|)+(2k+2+4(n-7))\geq 6n+2k-20$ by (\ref{eq4.1}).
If $|C(Q_{ij})\cap U|=1$ for any $i\in\{1,2\}$ and $j\in\{3,4\}$, we take $C(Q_{13})\cap U=\{w_1\}$ and $C(Q_{24})\cap U=\{w_2\}$. In fact, $w_1\in X_{123}\cup X_{134}\cup X_{1234}$ and $w_2\in X_{124}\cup X_{234}\cup X_{1234}$. When $w_1=w_2$, then $w_1=w_2\in X_{1234}$, thus $|PP(Q_{13})\cap U|\geq k-2$, which implies $d_{G}(w_1)\geq k+2$, and $2|E(G)|\geq (12+2|U_2|+3|U_3|+4|U_4|)+(k+2+4(n-6))\geq 6n+k-16$ by (\ref{eq4.1}).
When $w_1\neq w_2$, then $d_{G}(w_1)+d_{G}(w_2)\geq 2k+2$, and $2|E(G)|\geq 6n+2k-20$ by (\ref{eq4.1}).

Combining above arguments, we have $|E(G)|\geq3n-5$ for $k\geq6$, $|E(G)|\geq 3n-5>\lceil \frac{17n-32}{6} \rceil$ for $k=5$, $|E(G)|\geq 3n-6>\frac{5n-5}{2}$ for $k=4$, then (\ref{eq4.3}) holds.
{\hfill $\square$ \par}

Therefore, if $G[N(v)]\cong P_3\cup K_1$ or $G[N(v)]\cong 2K_2$, then (\ref{eq4.3}) holds by Claims $9$-$10$.

$\mathbf{Subcase~3.3:}$ $|E(G[N(v)])|=3$.

Then $G[N(v)]\not\cong K_{1,3}$ by Claim $3$, thus $G[N(v)]\cong P_4$ or $G[N(v)]\cong K_3\cup K_1$. If there exists a vertex $u\in U$ with $d_{G}(u)\geq k+1$, by (\ref{eq4.1}), then
\begin{align}
2|E(G)|
&\geq (14+2|U_2|+3|U_3|+4|U_4|)+(k+1+4(n-6))
\nonumber \\
&=6n+|U_3|+2|U_4|+k-19.
\label{eq4.7}
\qedhere
\end{align}

\noindent\textbf{Claim 11.}
Let $k\geq4$, $\delta(G)=4\leq\Delta(G)\leq n-2$ and $G[N(v)]\cong P_4$. Then (\ref{eq4.3}) holds.

\noindent\textbf{\textit{Proof of Claim 11}.}
Without loss of generality, let $v_1v_2, v_2v_3, v_3v_4\in E(G)$. Then
$X_{13}=X_{14}=X_{24}=\emptyset$ by Claim $1$, and $V(G)=N[v]\cup X_{12}\cup X_{23}\cup X_{34}\cup X_{123}\cup X_{124}\cup X_{134}\cup X_{234}\cup X_{1234}$ with $U_2=X_{12}\cup X_{23}\cup X_{34}$, $U_3=X_{123}\cup X_{124}\cup X_{134}\cup X_{234}$ and $U_4=X_{1234}$. If $k\geq5$, then $d_{G}(x)\geq 5$ for any $x\in U_4$ since $G+vx$ must contain a copy of $K_{2}\vee P_{k}$. By (\ref{eq4.2}), we have $|E(G)|\geq 7+2|U_2|+3|U_3|+4|U_4|+(|U_2|+0.5|U_3|+0.5|U_4|)=3n-8+0.5|U_3|+1.5|U_4|$ for $k\geq5$ and $|E(G)|\geq 3n-8+0.5|U_3|+|U_4|$ for $k\geq4$.

If $|U_4|\geq 2$, then $|E(G)|\geq 3n-5$ for $k\geq5$ and $|E(G)|\geq 3n-6$ for $k=4$.

If $|U_4|\leq1$, we consider the copy of $K_{2}\vee P_{k}$ in $G+v_1v_4$, say $Q$, we have $v\notin V(Q)$ by Claim $2$, and $C(Q)\subseteq\{v_1,v_4\}\cup X_{124}\cup X_{134}\cup X_{1234}$ by Corollary \ref{cor4-2}. Moreover,
$|X_{124}\cup X_{134}\cup X_{1234}|=|N_{Q}(v_1)\cap N_{Q}(v_4)\cap U|\geq 2$ by Claim $2$, which implies $|U_3|+|U_4|\geq 2$.

When $C(Q)=\{v_1,v_4\}$, then $PP(Q)\subseteq X_{124}\cup X_{134}\cup X_{1234}$, which implies $|U_3|+|U_4|\geq|PP(Q)|=k\geq4$, and thus $|E(G)|\geq 3n-5$ for $k\geq5$ and $|E(G)|\geq 3n-6$ for $k=4$.

When $C(Q)\neq\{v_1,v_4\}$, we take $u\in C(Q)\cap U$, then $d_{G}(u)\geq k+1$, and $|U|=n-5\geq |V(Q)\cap U|\geq k-1$, say, $n\geq8$.
If $|U_4|=1$, or $|U_4|=0$ with $|U_3|\geq 3$, then $|U_3|+2|U_4|\geq 3$, and thus $2|E(G)|\geq 6n+k-16$ by (\ref{eq4.7}).
If $|U_4|=0$ and $|U_3|=2$, then $|X_{124}\cup X_{134}|=2$ and $|X_{123}|=|X_{234}|=0$. Now we show $|X_{124}|=|X_{134}|=1$. Otherwise, if $|X_{124}|=2$ or $|X_{134}|=2$, then $G+v_1v_3$ or $G+v_2v_4$ contains no copy of $K_{2}\vee P_{k}$, a contradiction. Let $X_{124}=\{x_1\}$ and $X_{134}=\{x_2\}$. Then $d_{G}(x_i)\geq k+1$ for $i\in\{1,2\}$ by considering $G+v_1v_3$ and $G+v_2v_4$. It follows that $2|E(G)|\geq (14+2|U_2|+3|U_3|+4|U_4|)+(2k+2+4(n-7))\geq 6n+2k-20$ by (\ref{eq4.1}).

Combining above arguments, we have $|E(G)|\geq3n-5$ for $k\geq6$, $|E(G)|\geq 3n-5>\lceil \frac{17n-32}{6} \rceil$ for $k=5$, $|E(G)|\geq 3n-6>\frac{5n-5}{2}$ for $k=4$, then (\ref{eq4.3}) holds.
{\hfill $\square$ \par}

\noindent\textbf{Claim 12.}
Let $k\geq4$, $\delta(G)=4\leq\Delta(G)\leq n-2$ and $G[N(v)]\cong K_3\cup K_1$. Then (\ref{eq4.3}) holds.

\noindent\textbf{\textit{Proof of Claim 12}.}
Without loss of generality, let $v_1v_2,v_2v_3,v_3v_1\in E(G)$. Then
$X_{14}=X_{24}=X_{34}=\emptyset$ by Claim $1$, and $V(G)=N[v]\cup X_{12}\cup X_{13}\cup X_{23}\cup X_{123}\cup X_{124}\cup X_{134}\cup X_{234}\cup X_{1234}$ with $U_2=X_{12}\cup X_{13}\cup X_{23}$, $U_3=X_{123}\cup X_{124}\cup X_{134}\cup X_{234}$ and $U_4=X_{1234}$. Clearly, $|U|=n-5\geq 3$ by $d_{G}(v_4)\geq4$, say, $n\geq8$.

Let $Q$ be the a copy of $K_{2}\vee P_{k}$ in $G+v_1v_4$. Then $v\notin V(Q)$ by Claim $2$, and $C(Q)\subseteq\{v_1,v_4\}\cup X_{124}\cup X_{134}\cup X_{1234}$ by Corollary \ref{cor4-2}.

If $C(Q)=\{v_1,v_4\}$, then $PP(Q)\subseteq X_{124}\cup X_{134}\cup X_{1234}$, and thus $|U_3|+|U_4|\geq |PP(Q)|=k\geq4$. By (\ref{eq4.2}), $|E(G)|\geq 7+2|U_2|+3|U_3|+4|U_4|+(|U_2|+0.5|U_3|)=3n-8+0.5|U_3|+|U_4|$, thus $|E(G)|\geq 3n-5$ for $k\geq5$ and $|E(G)|\geq 3n-6$ for $k=4$.

If $C(Q)\neq\{v_1,v_4\}$, we take $u\in C(Q)\cap U$, then $d_{G}(u)\geq k+1$.
Since $G+v_iv_4$ must contain a copy of $K_{2}\vee P_{k}$, where $i\in\{1,2,3\}$, we have $|X_{124}\cup X_{134}\cup X_{1234}|\geq 2$, $|X_{124}\cup X_{234}\cup X_{1234}|\geq 2$ and $|X_{134}\cup X_{234}\cup X_{1234}|\geq 2$ by Claim $2$, thus $|U_3|+2|U_4|\geq 3$.
Therefore, $2|E(G)|\geq 6n+k-16$ by (\ref{eq4.7}).

Combining above arguments, we have $|E(G)|\geq3n-5$ for $k\geq6$, $|E(G)|\geq 3n-5>\lceil \frac{17n-32}{6} \rceil$ for $k=5$, $|E(G)|\geq 3n-6>\frac{5n-5}{2}$ for $k=4$, then (\ref{eq4.3}) holds.
{\hfill $\square$ \par}

Therefore, if $G[N(v)]\cong P_4$ or $G[N(v)]\cong K_3\cup K_1$, then (\ref{eq4.3}) holds by Claims $11$-$12$.

$\mathbf{Subcase~3.4:}$ $|E(G[N(v)])|=4$.

Then $G[N(v)]\cong C_4$ or $G[N(v)]\cong K_{1,3}^{+}$, where $K_{1,3}^{+}$ is a graph obtained by adding an edge to $K_{1,3}$.
If there exists a vertex $u\in U$ with $d_{G}(u)\geq 5$, by (\ref{eq4.1}), then we have
\begin{equation}\label{eq4.8}
2|E(G)|\geq (16+2|U_2|+3|U_3|+4|U_4|)+(5+4(n-6))=6n-13+|U_3|+2|U_4|.
\end{equation}

\noindent\textbf{Claim 13.}
Let $k\geq5$, $\delta(G)=4\leq\Delta(G)\leq n-2$ and $G[N(v)]\cong C_4$. Then (\ref{eq4.3}) holds.

\noindent\textbf{\textit{Proof of Claim 13}.}
Without loss of generality, let $v_1v_2, v_2v_3, v_3v_4, v_4v_1\in E(G)$. Then
$X_{13}=X_{24}=\emptyset$ by Claim $1$, and $V(G)=N[v]\cup X_{12}\cup X_{14}\cup X_{23}\cup X_{34}\cup X_{123}\cup X_{124}\cup X_{134}\cup X_{234}\cup X_{1234}$.

If $U_4\neq\emptyset$, then $d_{G}(x)\geq 5$ for any $x\in U_4$ since $G+vx$ must contain a copy of $K_{2}\vee P_{k}$, and thus $2|E(G)|\geq6n-11$ by (\ref{eq4.8}).

If $U_4=\emptyset$, similar to the proof of Claim $2$, we consider the center of the copy of $K_2\vee P_k$ in $G+v_1v_3$ and $G+v_2v_4$, and we have
$|X_{123}\cup X_{134}|\geq 1$ and $|X_{124}\cup X_{234}|\geq 1$, which implies $|U_3|\geq2$.
If there exists a vertex $u\in U$ such that $d_{G}(u)\geq 5$, then $2|E(G)|\geq6n-11$ by (\ref{eq4.8}).
If $d_{G}(u)=4$ for any $u\in U$, then we consider the copy of $K_2\vee P_k$ of $G+vw$, say $Q$, where $w\in U_3$. Without loss of generality, assume that $w\in X_{123}$. Clearly, $v,w\notin C(Q)$ by $d_{G}(v)=d_{G}(w)=4$ and $k\geq5$, then $C(Q)\subset \{v_1,v_2,v_3\}$ by Corollary \ref{cor4-2}. Without loss of generality, we take $C(Q)=\{v_1,v_2\}$, then $PP(Q)\backslash \{v\}\subseteq U$. Therefore, there exists a path $ww_1\cdots w_{k-2}$ in $G[U]$ with $N(w_i)\cap N(v)=\{v_1,v_2\}$ by $d_{G}(w_i)=4$ for $1\leq i\leq k-3$, which implies that there exists a unique path with length $2$ between $w_1$ and $v_4$ in $G$, say, $w_1v_1v_4$, a contradiction with Proposition \ref{pro4-1}.

Combining above arguments, we have $|E(G)|\geq3n-5$ for $k\geq6$, $|E(G)|\geq 3n-5>\lceil \frac{17n-32}{6} \rceil$ for $k=5$, then (\ref{eq4.3}) holds.
{\hfill $\square$ \par}

\noindent\textbf{Claim 14.}
Let $k\geq5$, $\delta(G)=4\leq\Delta(G)\leq n-2$ and $G[N(v)]\cong K_{1,3}^{+}$. Then (\ref{eq4.3}) holds.

\noindent\textbf{\textit{Proof of Claim 14}.}
Without loss of generality, let $v_1v_2, v_2v_3, v_3v_1, v_3v_4\in E(G)$. Then
$X_{14}=X_{24}=\emptyset$ by Claim $1$, and $V(G)=N[v]\cup X_{12}\cup X_{13}\cup X_{23}\cup X_{34}\cup X_{123}\cup X_{124}\cup X_{134}\cup X_{234}\cup X_{1234}$. Thus $|E(G)|\geq 8+2|U_2|+3|U_3|+4|U_4|+(|U_2|+0.5|U_3|)=3n-7+0.5|U_3|+|U_4|$ by (\ref{eq4.2}).

If $|U_4|\geq 2$, or $|U_4|=1$ with $|U_3|\geq 1$, or $|U_4|=0$ with $|U_3|\geq 3$, then $|E(G)|\geq 3n-5$.

If $|U_4|=1$ and $|U_3|=0$, we take $U_4=\{w\}$, then $d_{G}(w)\geq 5$ since $G+vw$ contains a copy of $K_{2}\vee P_{k}$, and thus $2|E(G)|\geq 6n-11$ by (\ref{eq4.8}).

If $|U_4|=0$ and $|U_3|\leq 2$, then $|U_3|\geq 1$ since $G+v_1v_4$ must contain a copy of $K_{2}\vee P_{k}$.
When $d_{G}(u)\geq 6$ for some $u\in U$, then $2|E(G)|\geq
6n-12+|U_3|\geq 6n-11$ by (\ref{eq4.1}).
When $4\leq d_{G}(u)\leq 5$ for any $u\in U$, similar to the proof of Claim $2$, we consider the center of the copy of $K_2\vee P_k$ in $G+v_1v_4$ and $G+v_2v_4$, and we have $|X_{123}\cup X_{134}|\geq 1$ and $|X_{123}\cup X_{234}|\geq 1$. Now we only need to consider $|X_{123}\cup X_{134}\cup X_{234}|=2$ or $|X_{123}\cup X_{134}\cup X_{234}|=1$ by $|U_3|\leq 2$.

When $|X_{123}\cup X_{134}\cup X_{234}|=2$, we have $|X_{124}|=0$, and thus $X_{12}\neq\emptyset$. Otherwise, $d_{G}(v_3)=n-1$, a contradicts with $\Delta(G)\leq n-2$.
If $d_{G}(u)=5$ for some $u\in U$, then $2|E(G)|\geq 6n-11$ by (\ref{eq4.8}).
If $d_{G}(u)=4$ for any $u\in U$, we take $x\in X_{12}$. By Proposition \ref{pro4-1}, there exist at least two paths with length $2$ between $x$ and $v_4$, thus $N(x)\cap U\subseteq X_{34}\cup X_{134}\cup X_{234}$. Since $G+vx$ contains a copy of $K_{2}\vee P_{k}$, say $Q$, we have $C(Q)=\{v_1,v_2\}$ by $d_{G}(v)=d_{G}(x)=4$ and $k\geq5$, $PP(Q)\subseteq\{x,v,v_3\}\cup X_{12}\cup X_{123}$,
and thus $d_{Q}(x)=3$ by $N(x)\cap U\cap C(Q)=\emptyset$. Therefore, there exists a path $P=v_3y_1y_2\cdots y_{k-3}$ in $G$ with $y_1\in X_{123}$ and $y_2\in X_{12}\cup X_{123}$ (by $y_2v_1,y_2v_2\in E(G)$ and $X_{124}=X_{1234}=\emptyset$).  However, there exists a unique path with length $2$ between $y_1$ and $v_4$ in $G$, say, $y_1v_3v_4$, a contradiction with Proposition \ref{pro4-1}.

When $|X_{123}\cup X_{134}\cup X_{234}|=1$, we have $|X_{123}|=1$, $|X_{134}|=0$ and $|X_{234}|=0$ by $|X_{123}\cup X_{134}|\geq 1$ and $|X_{123}\cup X_{234}|\geq 1$. Let $X_{123}=\{y\}$. Since $G+vy$ contains a copy of $K_{2}\vee P_{k}$, say $Q'$, we have $v\notin C(Q')$ by $d_{G}(v)=4$ and $k\geq 5$, thus $C(Q')\subset\{v_1,v_2,v_3,y\}$ by Corollary \ref{cor4-2}. Now we show $y\in C(Q')$. If not, $C(Q')\subset\{v_1,v_2,v_3\}$, without loss of generality, let $C(Q')=\{v_1,v_2\}$. Then $PP(Q')\subseteq \{v_3,v,y\}\cup X_{12}\cup X_{124}$.
If $v_3\notin PP(Q')$, then $G[(V(Q')\setminus\{v\})\cup\{v_3\}]$ and thus $G$ contains a copy of $K_{2}\vee P_{k}$, a contradiction. If $v_3\in PP(Q')$, then $Q'-vy+v_3y$ is a copy of $K_{2}\vee P_{k}$ in $G$, a contradiction.
Therefore, $d_{G}(y)\geq k+1$, and $2|E(G)|\geq 16+2|U_2|+3|U_3|+k+1+4(n-6)\geq 6n+k-16$.

Combining above arguments, we have $|E(G)|\geq3n-5$ for $k\geq6$, $|E(G)|\geq 3n-5>\lceil \frac{17n-32}{6} \rceil$ for $k=5$, then (\ref{eq4.3}) holds.
{\hfill $\square$ \par}

Therefore, if $G[N(v)]\cong C_4$ or $G[N(v)]\cong K_{1,3}^{+}$, then (\ref{eq4.3}) holds by Claims $13$-$14$.

$\mathbf{Subcase~3.5:}$ $5\leq|E(G[N(v)])|\leq6$.

\noindent\textbf{Claim 15.}
Let $k\geq4$, $\delta(G)=4\leq\Delta(G)\leq n-2$ and $5\leq|E(G[N(v)])|\leq6$. Then (\ref{eq4.3}) holds.

\noindent\textbf{\textit{Proof of Claim 15}.}
If $|E(G[N(v)])|=5$, then $G[N(v)]\cong C_{4}^{+}$, where $C_{4}^{+}$ is a graph obtained by adding an edge to $C_4$. Without loss of generality, let $v_1v_2, v_2v_3, v_3v_4, v_4v_1, v_1v_3\in E(G)$. Then
$X_{24}=\emptyset$ by Claim $1$, and $|U_3|+|U_4|\geq1$ since there exists a copy of $K_{2}\vee P_{k}$ in $G+v_2v_4$ for $k\geq4$. Therefore, $|E(G)|\geq 9+2|U_2|+3|U_3|+4|U_4|+(|U_2|+0.5|U_3|)=3n-6+0.5|U_3|+|U_4|\geq 3n-5$ by (\ref{eq4.2}).

If $|E(G[N(v)])|=6$, then $G[N[v]]\cong K_5$, and thus $|E(G)|\geq 10+2|U_2|+3|U_3|+4|U_4|+(|U_2|+0.5|U_3|)\geq 3n-5$ by (\ref{eq4.2}) and $|U_2|+|U_3|+|U_4|=n-5$.

Combining above arguments, we have $|E(G)|\geq3n-5$ for $k\geq6$, $|E(G)|\geq 3n-5>\lceil \frac{17n-32}{6} \rceil$ for $k=5$, $|E(G)|\geq 3n-5>\frac{5n-5}{2}$ for $k=4$, then (\ref{eq4.3}) holds.
{\hfill $\square$ \par}

Clearly, if $5\leq|E(G[N(v)])|\leq6$, then (\ref{eq4.3}) holds.

$\mathbf{Case~4:}$ $\delta(G)=5$.

Then $n\geq 12$ by $n\geq a_{k}+2$ and $k\geq 6$.
By (\ref{eq4.1}), we have
\begin{align}
2|E(G)|
&\geq 2(5+|E(G[N(v)])|)+2|U_2|+3|U_3|+4|U_4|+5|U_5|+5(n-6)
\nonumber \\
&=7n+2|E(G[N(v)])|-32+|U_3|+2|U_4|+3|U_5|.
\label{eq4.9}
\qedhere
\end{align}

By Claim $1$, we have $|E(G[N(v)])|\neq0$, then $1\leq|E(G[N(v)])|\leq10$.

If $|E(G[N(v)])|\geq 5$, then $2|E(G)|\geq 7n-22\geq 6n-10$ by (\ref{eq4.9}), (\ref{eq4.3}) holds.

If $|E(G[N(v)])|=4$, then $G[N(v)]\not\cong K_{1,4}$ by Claim $3$, and thus there exist $v_i,v_j\in N(v)$ with $v_iv_j\notin E(G)$ and $N(v_i)\cap N(v_j)\cap N(v)=\emptyset$ for some $1\leq i<j\leq5$. It follows that $X_{ij}=\emptyset$ by Claim $1$ and $|N(v_i)\cap N(v_j)\cap U|\geq 2$ by Claim $2$, which implies $|U_3\cup U_4\cup U_5|\geq 2$. Thus $2|E(G)|\geq 7n-22\geq 6n-10$ by (\ref{eq4.9}), (\ref{eq4.3}) holds.

In the following, we consider $1\leq|E(G[N(v)])|\leq3$.

$\mathbf{Subcase~4.1:}$ $|E(G[N(v)])|=1$.

Without loss of generality, let $v_1v_2\in E(G)$.
Then $V(G)=N[v]\cup X_{12}\cup X_{123}\cup X_{124}\cup X_{125}\cup X_{1234}\cup X_{1235}\cup X_{1245}\cup X_{12345}$ by Claim $1$. Since $d_{G}(v_i)\geq \delta(G)=5$ for $i\in\{3,4,5\}$, we have $|U_3|+2|U_4|+3|U_5|=|E[\{v_3,v_4,v_5\},U]|=d_{G}(v_3)+d_{G}(v_4)+d_{G}(v_5)-3\geq 12$ by $N(v_3)\cap U=X_{123}\cup X_{1234}\cup X_{1235}\cup X_{12345}$, $N(v_4)\cap U=X_{124}\cup X_{1234}\cup X_{1245}\cup X_{12345}$ and $N(v_5)\cap U=X_{125}\cup X_{1235}\cup X_{1245}\cup X_{12345}$. Thus $2|E(G)|\geq 7n-18\geq 6n-11$ by (\ref{eq4.9}), (\ref{eq4.3}) holds.

$\mathbf{Subcase~4.2:}$ $|E(G[N(v)])|=2$.

If $G[N(v)]\cong P_3\cup2K_1$, without loss of generality, let $v_1v_2,v_2v_3\in E(G)$. Then $V(G)=N[v]\cup X_{12}\cup X_{23}\cup X_{123}\cup X_{124}\cup X_{125}\cup X_{234}\cup X_{235}\cup X_{1234}\cup X_{1235}\cup X_{1245}\cup X_{2345}\cup X_{12345}$ by Claim $1$. Since $G+v_4v_5$, $G+v_1v_4$ and $G+v_3v_4$ must contain a copy of $K_2\vee P_{k}$, then $|X_{1245}\cup X_{2345}\cup X_{12345}|\geq 2$, $|X_{124}\cup X_{1234}\cup X_{1245}\cup X_{12345}|\geq 2$ and $|X_{234}\cup X_{1234}\cup X_{2345}\cup X_{12345}|\geq 2$ by Claim $2$, and thus $|U_3|+2|U_4|+3|U_5|\geq 5$.

If $G[N(v)]\cong 2K_2\cup K_1$, without loss of generality, let $v_1v_2,v_3v_4\in E(G)$.
Then $V(G)=N[v]\cup X_{12}\cup X_{34}\cup X_{123}\cup X_{124}\cup X_{125}\cup X_{134}\cup X_{234}\cup X_{345}\cup X_{1234}\cup X_{1235}\cup X_{1245}\cup X_{1345}\cup X_{2345}\cup X_{12345}$ by Claim $1$. Since $G+v_2v_3$, $G+v_1v_5$ and $G+v_4v_5$ must contain a copy of $K_2\vee P_{k}$, then $|X_{123}\cup X_{234}\cup X_{1234}\cup X_{1235}\cup X_{2345}\cup X_{12345}|\geq 2$, $|X_{125}\cup X_{1235}\cup X_{1245}\cup X_{1345}\cup X_{12345}|\geq 2$ and $|X_{345}\cup X_{1245}\cup X_{1345}\cup X_{2345}\cup X_{12345}|\geq 2$ by Claim $2$, and thus $|U_3|+2|U_4|+3|U_5|\geq 5$.

Therefore, $2|E(G)|\geq7n-23\geq 6n-11$ by (\ref{eq4.9}), (\ref{eq4.3}) holds.

$\mathbf{Subcase~4.3:}$ $|E(G[N(v)])|=3$.

Then $G[N(v)]\in\{K_{1,3}\cup K_1,K_3\cup 2K_1,P_4\cup K_1,P_3\cup P_2\}$, and there exists $v_i,v_j\in N(v)$ with $v_iv_j\notin E(G)$ and $N(v_i)\cap N(v_j)\cap N(v)=\emptyset$ for some $1\leq i<j\leq 5$.
Since $G+v_iv_j$ contains a copy of $K_2\vee P_k$, say $Q$, we have $C(Q)\subset\{v_i,v_j\}\cup(N(v_i)\cap N(v_j))\subseteq \{v_i,v_j\}\cup U_3\cup U_4\cup U_5$ by Corollary \ref{cor4-2} and $v\notin C(Q)$.
Then $|N_{Q}(v_i)\cap N_{Q}(v_j)\cap U|\geq 2$ by Claim $2$, which implies
$|U_3\cup U_4\cup U_5|\geq 2$.

If $C(Q)=\{v_i,v_j\}$, then $PP(Q)\subseteq U_3\cup U_4\cup U_5$ and $|U_3\cup U_4\cup U_5|\geq k\geq6$, which implies $|U_3|+2|U_4|+3|U_5|\geq 6$,
and $2|E(G)|\geq7n-20\geq 6n-11$ by (\ref{eq4.9}), (\ref{eq4.3}) holds.

If $C(Q)\neq\{v_i,v_j\}$, we take $u\in C(Q)\cap(U_3\cup U_4\cup U_5)$, then $d_{G}(u)\geq k+1$, and $2|E(G)|\geq 16+2|U_2|+3|U_3|+4|U_4|+5|U_5|+k+1+5(n-7)=7n+k-30+|U_3|+2|U_4|+3|U_5|\geq 7n-22\geq 6n-11$ by (\ref{eq4.1}) and $|U_3\cup U_4\cup U_5|\geq 2$, (\ref{eq4.3}) holds.
{\hfill $\square$ \par}

\subsection{$\boldsymbol{k=3}$ with $\boldsymbol{3\leq \delta(G)\leq\Delta(G)\leq n-2}$}
In this part, we show that (\ref{eq4.3}) holds, with equality if and only if $G\cong H_9$ and $n=6$.

$\mathbf{Case~1:}$ $\delta(G)\geq 5$.

Then $|E(G)|\geq \frac{5n}{2}>\lfloor \frac{5n-8}{2} \rfloor$, (\ref{eq4.3}) holds.

$\mathbf{Case~2:}$ $\delta(G)=3$.

By Claim $4$, we have $|E(G[N(v)])|=1$ or $3$.

If $|E(G[N(v)])|=1$, without loss of generality, let $v_1v_2\in E(G)$. Then
$X_{13}=X_{23}=\emptyset$ and $U=X_{12}\cup X_{123}$. Now we show $X_{12}=\emptyset$. Otherwise, let $x\in X_{12}$. By Proposition \ref{pro4-1}, there exist at least two paths with length $2$ between $x$ and $v_3$, which implies $|N(x)\cap X_{123}|\geq2$, and thus $G$ contains a copy of $K_2\vee P_3$, say $Q$, where $C(Q)=\{v_1,v_2\}$ and $PP(Q)\subseteq\{x\}\cup (N(x)\cap X_{123})$, a contradiction.
Therefore, $U=X_{123}$. Clearly, $G[U]$ contains no copy of $P_3$. Now we show
$G[U]$ has no isolated vertex. If not, there exists a vertex $u\in U$ with $d_{G}(u)=3$, but $G+vu$ contains no copy of $K_2\vee P_3$, a contradiction.
Therefore, $G[U]\cong\frac{n-4}{2}K_2$ and $n(\geq6)$ is even, and thus $|E(G)|=4+3(n-4)+\frac{n-4}{2}=\frac{7}{2}n-10\geq \lfloor \frac{5n-8}{2} \rfloor $, (\ref{eq4.3}) holds, with equality if and only if $n=6$ and $G\cong H_9$.

If $|E(G[N(v)])|=3$, then $X_{123}=\emptyset$. If not, then $G[N[v]\cup X_{123}]$ contains a copy of $K_2\vee P_3$, a contradiction. By $\Delta(G)\leq n-2$, we have $X_{12}\neq\emptyset$, $X_{13}\neq\emptyset$ and $X_{23}\neq\emptyset$, which implies $n\geq 7$. By Claim $6$, $G[U]$ is connected, and thus $|E(G)|=6+2(n-4)+(|U|-1)=3n-7>\lfloor \frac{5n-8}{2} \rfloor$, (\ref{eq4.3}) holds.

$\mathbf{Case~3:}$ $\delta(G)=4$.

If $5\leq |E(G[N(v)])|\leq6$, then $G[N(v)]\cong K_4$ or $C_4^{+}$, and thus $G$ contains a copy of $K_2\vee P_3$, a contradiction.

If $|E(G[N(v)])|=1$, without loss of generality, let $v_1v_2\in E(G)$. Similar to the proof of Case $2$, we have $X_{12}=\emptyset$ and $U=X_{123}\cup X_{124}\cup X_{1234}$ with $|U|=n-5\geq3$ by $d_{G}(v_3)\geq\delta(G)=4$.
Since $v_3v_4\notin E(G)$ and $N(v_3)\cap N(v_4)\cap N(v)=\emptyset$, we have $|U_4|=|N(v_3)\cap N(v_4)\cap U|\geq2$ by Claim $2$.
Similar to Case $2$, we can show $G[U]\cong \frac{n-5}{2}K_2$ and $n(\geq9)$ is odd, thus
$|E(G)|=5+3|U_3|+4|U_4|+\frac{n-5}{2}\geq \frac{7n-21}{2}>\lfloor \frac{5n-8}{2} \rfloor$, (\ref{eq4.3}) holds.

If $|E(G[N(v)])|=2$, then $G[N(v)]\cong P_3\cup K_1$ or $G[N(v)]\cong 2K_2$.

When $G[N(v)]\cong P_3\cup K_1$, without loss of generality, let $v_1v_2,v_2v_3\in E(G)$.
Then $|U|=n-5\geq3$ by $d_{G}(v_4)\geq4$.
Since $G+v_1v_4$ and $G+v_3v_4$ must contain a copy of $K_2\vee P_3$, we have $|X_{124}\cup X_{1234}|\geq2$ and $|X_{234}\cup X_{1234}|\geq2$ by Claim $2$, thus $0.5|U_3|+|U_4|\geq 2$. Therefore, $|E(G)|\geq 6+2|U_2|+3|U_3|+4|U_4|+(|U_2|+0.5|U_3|)\geq3n-7>\lfloor \frac{5n-8}{2}\rfloor$ by (\ref{eq4.2}), (\ref{eq4.3}) holds.

When $G[N(v)]\cong 2K_2$, without loss of generality, let $v_1v_2,v_3v_4\in E(G)$.
Then $|U|=n-5\geq2$ by $d_{G}(v_1)\geq4$. Since $G+v_1v_4$ and $G+v_2v_3$ must contain a copy of $K_2\vee P_3$, we have $|X_{124}\cup X_{134}\cup X_{1234}|\geq2$ and $|X_{123}\cup X_{234}\cup X_{1234}|\geq2$ by Claim $2$, thus $0.5|U_3|+|U_4|\geq 2$. Therefore, $|E(G)|>\lfloor \frac{5n-8}{2}\rfloor$ by (\ref{eq4.2}), (\ref{eq4.3}) holds.

If $|E(G[N(v)])|=3$, then $G[N(v)]\not\cong K_{1,3}$ by Claim $3$, and thus $G[N(v)]\cong P_4$ or $G[N(v)]\cong K_3\cup K_1$. Thus $n\geq 7$, and there exist two vertices $v_i$ and $v_j$ such that $v_iv_j\notin E(G)$ and $N(v_i)\cap N(v_j)\cap N(v)=\emptyset$.
By Claim $2$, $|N(v_i)\cap N(v_j)\cap U|\geq2$, which implies $|U_3\cup U_4|\geq2$.
Therefore, $|E(G)|\geq 7+2|U_2|+3|U_3|+4|U_4|+(|U_2|+0.5|U_3|)\geq 3n-7>\lfloor \frac{5n-8}{2} \rfloor$ by (\ref{eq4.2}), (\ref{eq4.3}) holds.

If $|E(G[N(v)])|=4$, then $G[N(v)]\cong C_4$ or $G[N(v)]\cong K_{1,3}^{+}$, and thus $n\geq 6$ by $\delta(G)=4$. If $n=6$, then $G[N(v)]\cong C_4$, $|U|=|U_4|=1$, it follows that $|E(G)|=12>\lfloor \frac{5n-8}{2} \rfloor$. If $n\geq 7$, then $|E(G)|\geq 8+2|U_2|+3|U_3|+4|U_4|+(|U_2|+0.5|U_3|)\geq 3n-7>\lfloor \frac{5n-8}{2} \rfloor$ by (\ref{eq4.2}), (\ref{eq4.3}) holds.
{\hfill $\square$ \par}

\subsection{$\boldsymbol{k=4}$ with $\boldsymbol{3\leq \delta(G)\leq\Delta(G)\leq n-2}$}
In this part, we show that (\ref{eq4.3}) holds, say, $|E(G)|\geq\lfloor \frac{5n-8}{2} \rfloor+\sigma_{n-1}$, with equality if and only if $G\cong H_{10}$ and $n=7$, where $\sigma_{n-1}=2$ if $n$ is odd, and $0$ if $n$ is even.

$\mathbf{Case~1:}$ $\delta(G)\geq 5$.

Then $|E(G)|\geq \frac{5n}{2}>\lfloor \frac{5n-8}{2} \rfloor+\sigma_{n-1}$, (\ref{eq4.3}) holds.

$\mathbf{Case~2:}$ $\delta(G)=3$.

By Claim $4$, we have $|E(G[N(v)])|=1$ or $3$.

If $|E(G[N(v)])|=1$, without loss of generality, let $v_1v_2\in E(G)$. By Claim $5$, we have $|E(G)|\geq 3n-6\geq\lfloor \frac{5n-8}{2} \rfloor+\sigma_{n-1}$, with equality if and only if $n=7$.

When $n=7$ and $X_{12}\neq\emptyset$, then $|X_{123}|=2$ by $d_{G}(v_3)\geq\delta(G)=3$, thus $G\cong H_{10}$.

When $n=7$ and $X_{12}=\emptyset$, then $G[X_{123}]$ contains $P_3$ from the proof of Claim $5$. If $G[X_{123}]\cong P_3$, then there exists a vertex $x\in X_{123}$ such that $d_{G}(x)=5$. However $G+vx$ contains no copy of $K_{2}\vee P_{4}$, a contradiction. If $G[X_{123}]\cong K_3$, then $G[V(G)\backslash\{v\}]$ contains a copy of $K_{2}\vee P_{4}$, a contradiction.

If $|E(G[N(v)])|=3$, then $X_{12}\neq\emptyset$, $X_{13}\neq\emptyset$ and $X_{23}\neq\emptyset$ by $\Delta(G)\leq n-2$.

When $X_{123}=\emptyset$, then $G[U]$ is connected by Claim $6$. For $1\leq i<j\leq3$, $|X_{ij}|\geq 2$ since $G+vx$ must contain a copy of $K_{2}\vee P_{4}$, where $x\in X_{ij}$. Thus $n\geq 10$ and $|E(G)|\geq 6+2|U_2|+|U|-1=3n-7>\lfloor \frac{5n-8}{2} \rfloor+\sigma_{n-1}$, (\ref{eq4.3}) holds.

When $X_{123}\neq\emptyset$, then $d_{G}(u)=3$ for any $u\in X_{123}$. Otherwise, $G[N[u]\cup N[v]]$ contains a copy of $K_{2}\vee P_{4}$, a contradiction.
Moreover, $|X_{123}|\geq 2$. If not, $X_{123}=\{x\}$, but $G+vx$ contains no copy of $K_{2}\vee P_{4}$, a contradiction. Thus $n\geq9$, and $G[U_2]$ is connected by Claim $6$.

If $n=9$, then $|X_{12}|=|X_{13}|=|X_{23}|=1$ and $G[U_2]\cong P_3$ or $K_3$. When $G[U_2]\cong P_3$, we take $x_1x_2,x_2x_3\in E(G)$, where $\{x_1\}=X_{12}$, $\{x_2\}=X_{13}$ and $\{x_3\}=X_{23}$. However, $G+x_1x_3$ contains no copy of $K_{2}\vee P_{4}$, a contradiction. It follows that $G[U_2]\cong K_3$ and $|E(G)|=21>20=\lfloor \frac{5n-8}{2} \rfloor+\sigma_{n-1}$, (\ref{eq4.3}) holds.

If $n\geq 10$, then $|E(G)|\geq 6+2|U_2|+3|U_3|+|U_2|-1=3n-7>\lfloor \frac{5n-8}{2} \rfloor+\sigma_{n-1}$ by (\ref{eq4.2}), (\ref{eq4.3}) holds.

$\mathbf{Case~3:}$ $\delta(G)=4$.

If $|E(G[N(v)])|\in\{1,2,3,5,6\}$, then (\ref{eq4.3}) holds by Claims $8$-$12$ and Claim $15$.

If $|E(G[N(v)])|=4$, then $G[N(v)]\cong C_4$ or $G[N(v)]\cong K_{1,3}^{+}$, and $|E(G)|\geq 8+2|U_2|+3|U_3|+4|U_4|+(|U_2|+0.5|U_3|)=3n-7+0.5|U_3|+|U_4|$ by (\ref{eq4.2}).

If $|U_4|\geq2$, or $|U_4|=1$ with $|U_3|\geq1$, then $|E(G)|\geq 3n-5>\lfloor\frac{5n-8}{2} \rfloor+\sigma_{n-1}$, (\ref{eq4.3}) holds.

If $|U_3|+|U_4|=1$, then $n\neq7$. Otherwise, $n=7$ and $|U_2|=1$. Let $x\in U_2$. Then $d_{G}(x)\leq3<\delta(G)=4$, a contradiction.

If $|U_4|=1$ and $|U_3|=0$, then $|E(G)|\geq 3n-6>\lfloor\frac{5n-8}{2} \rfloor+\sigma_{n-1}$ by $n\neq7$, (\ref{eq4.3}) holds.

If $|U_4|=0$, we show (\ref{eq4.3}) holds by the following two subcases.

Let $G[N(v)]\cong C_4$ and $v_1v_2,v_2v_3,v_3v_4, v_4v_1\in E(G)$.
Similar to the proof of Claim $2$, we consider the center of the copy of $K_2\vee P_4$ in $G+v_1v_3$ and $G+v_2v_4$, then $|X_{123}\cup X_{134}|\geq 1$ and $|X_{124}\cup X_{234}|\geq 1$, which implies $|U|\geq|U_3|\geq2$. Now we show $|U|=n-5\geq3$. If not, then $|U|=|U_3|=2$, $|X_{123}\cup X_{134}|=1$ and $|X_{124}\cup X_{234}|=1$. When $|X_{123}|=1$, then $d_{G}(v_4)=\delta(G)=4$ and $|E(G[N(v_4)])=3$, a contradiction with the choice of $v$. When $|X_{134}|=1$, then $d_{G}(v_2)=\delta(G)=4$ and $|E(G[N(v_2)])=3$, a contradiction with the choice of $v$. Thus $|E(G)|\geq 3n-6>\lfloor\frac{5n-8}{2} \rfloor+\sigma_{n-1}$ by $n\geq8$, (\ref{eq4.3}) holds.

Let $G[N(v)]\cong K_{1,3}^{+}$ and $v_1v_2,v_2v_3, v_3v_1,v_3v_4\in E(G)$. Then $|U|=n-5\geq 2$ by $d_{G}(v_4)\geq 4$.
Similar to the proof of Claim $2$, we consider the center of the copy of $K_2\vee P_4$ in $G+v_1v_4$, we have $|X_{123}\cup X_{124}\cup X_{134}|\geq1$, which implies $|U_3|\geq1$.

Now we show $|U|\geq3$ and $n\geq8$. If $|U|=2$, then $|U_3|=2$ and $|U_2|=0$. Otherwise, $|U_3|+|U_4|=1$, which implies $n\neq7$, a contradiction with $n=5+|U|=7$.
On the other hand, $|X_{124}\cup X_{134}\cup X_{234}|\geq2$ by $d_{G}(v_4)\geq4$, $v_3v_4,vv_4\in E(G)$ and $|U_2|=0$. Then we have $|X_{124}\cup X_{134}\cup X_{234}|=2$ and $|X_{123}|=0$.

Clearly, $|X_{234}|\neq2$ by $|X_{124}\cup X_{134}|\geq1$ and $|U_3|=2$.
When $|X_{124}|=2$, then $d_{G}(v_4)=\delta(G)=4$ and $|E(G[N(v_4)])=2$, a contradiction with the choice of $v$. When $|X_{134}|=2$, then $d_{G}(v_2)=3$, a contradiction with $\delta(G)=4$. When $|X_{124}|=|X_{134}|=1$, then $d_{G}(v_4)=4$ and $|E(G[N(v_4)])=3$, a contradiction with the choice of $v$. When $|X_{124}|=|X_{234}|=1$, then $d_{G}(v_4)=4$ and $|E(G[N(v_4)])=3$, a contradiction with the choice of $v$. When $|X_{134}|=|X_{234}|=1$, we take $X_{234}=\{x\}$, then
$d_{G}(x)=4$ and $|E(G[N(x)])=3$, a contradiction with the choice of $v$.

Thus $|U|\geq3$, $n\geq8$, and $|E(G)|\geq 3n-6>\lfloor\frac{5n-8}{2} \rfloor+\sigma_{n-1}$, (\ref{eq4.3}) holds.
{\hfill $\square$ \par}

\subsection{$\boldsymbol{k=5}$ with $\boldsymbol{3\leq \delta(G)\leq\Delta(G)\leq n-2}$}
In this part, we show that (\ref{eq4.3}) holds strictly.

$\mathbf{Case~1:}$ $\delta(G)\geq 6$.

Then $|E(G)|\geq 3n>\lceil \frac{17n-32}{6} \rceil$, (\ref{eq4.3}) holds.

$\mathbf{Case~2:}$ $\delta(G)=3$.

By Claim $4$, we have $|E(G[N(v)])|=1$ or $3$.

If $|E(G[N(v)])|=1$, then (\ref{eq4.3}) holds by Claim $5$.

If $|E(G[N(v)])|=3$, then $X_{12}\neq\emptyset$, $X_{13}\neq\emptyset$ and $X_{23}\neq\emptyset$ by $\Delta(G)\leq n-2$.

When $X_{123}=\emptyset$, then for $1\leq i<j\leq3$, $|X_{ij}|\geq 2$ and $G[X_{ij}]$ has no isolated vertex by considering $G+vx$, where $x\in X_{ij}$, which implies $n\geq 10$.
For any $x_1\in X_{ij}$ and $x_2\in U\backslash X_{ij}$, $d_{G[U]}(x_1,x_2)\leq 2$ by Proposition \ref{pro4-1}. Thus $\sum_{u\in U}d(u)\geq 4(n-4)$, and $2|E(G)|\geq 12+2|U_2|+4(n-4)=6n-12>2\lceil \frac{17n-32}{6}\rceil$ by (\ref{eq4.1}), (\ref{eq4.3}) holds.

When $X_{123}\neq\emptyset$, for any $x\in X_{123}=U_3$, $Q$ is the copy of $K_2\vee P_5$ in $G+vx$. Then $v\in PP(Q)$ by $d_{G}(v)=3$, $x\in C(Q)$ since $G$ contains no copy of $K_2\vee P_5$, and $d_{G}(x)\geq5$ by $d_{G+vx}(x)\geq k+1$.
Furthermore, $N[v]\subset V(Q)$, $|V(Q)\cap X_{123}|=1$ and $|PP(Q)\cap U_2|=2$
since $G$ contains no copy of $K_2\vee P_5$. Let $y_1,y_2\in PP(Q)\cap U_2$. Then $G[\{w,y_1,y_2\}]\cong K_3$, and $|N(y_1)\cap N(y_2)\cap N(v)|=1$ since $G$ contains no copy of $K_2\vee P_5$, a contradiction.

Now we show $\sum_{u\in U}d(u)\geq 4|U_2|+5|U_3|$.

If $d_{G}(u)\geq 4$ for any $u\in U_2$, then the conclusion holds.

If $d_{G}(u)=3$ for some $u\in U_2$, without loss of generality, let $u\in X_{12}$ and $w$ be a neighbor of $u$ in $U$. Then $wz\in E(G)$ for any $z\in X_{13}\cup X_{23}$ since $uz\notin E(G)$ and there exists a path with length $2$ between $u$ and $z$ in $G[U]$ by Proposition \ref{pro4-1}.

Now we show $w\notin X_{13}\cup X_{23}$.
If $w\in X_{13}$, then $G+uy$ contains no copy of $K_2\vee P_5$ with $y\in X_{23}$, a contradiction.
If $w\in X_{23}$, we obtain a contradiction similarly by taking $y\in X_{13}$. Therefore, $w\in X_{12}\cup X_{123}$.

When $w\in X_{12}$, then $d_{G}(w)\geq 3+|X_{13}|+|X_{23}|\geq 5$. We can assign one degree of $w$ to $u$. Moreover, $d_{G}(z)\geq4$ for any $z\in X_{13}\cup X_{23}$. Otherwise, if $d_{G}(z)=3$ for some $z\in X_{13}\cup X_{23}$, by above arguments, $w\notin X_{12}$, a contradiction. For any $u_1\in X_{12}\backslash\{u\}$ with $u_1w\in E(G)$ and $d_{G}(u_1)=3$, it can be viewed as assigning one degree of $w$ to $u_1$.

When $w\in X_{123}$, then $d_{G}(u)+d_{G}(w)\geq 7+|X_{13}|+|X_{23}|\geq9$. We can assign one degree of $w$ to $u$. By considering $G+vw$, there exist $y_1,y_2\in U_2$ such that $G[\{w,y_1,y_2\}]\cong K_3$, thus $d_{G}(y_i)\geq4$ for $i\in\{1,2\}$. Therefore, for any $u_2\in U_2\backslash\{u\}$ with $u_2w\in E(G)$ and $d_{G}(u_2)=3$, it can be viewed as assigning one degree of $w$ to it.

Therefore, $\sum_{u\in U}d(u)\geq 4|U_2|+5|U_3|$, and $2|E(G)|= 12+2|U_2|+3|U_3|+\sum_{u\in U}d(u)\geq 12+6(n-4)+2|U_3|\geq6n-10>2\lceil \frac{17n-32}{6} \rceil$ by (\ref{eq4.2}), (\ref{eq4.3}) holds.

$\mathbf{Case~3:}$ $\delta(G)=4$.

By the discussion of $G[N(v)]$ and Claims $8$-$15$, (\ref{eq4.3}) holds.

$\mathbf{Case~4:}$ $\delta(G)=5$.

We know $n\geq a_{k}+2=7$. Recall (\ref{eq4.9}),
$2|E(G)|\geq 7n+2|E(G[N(v)])|-32+|U_3|+2|U_4|+3|U_5|$.

If $1\leq |E(G[N(v)])|\leq 2$, then $\delta(G[N(v)])=0$, and thus $n\geq 10$ by $d_{G}(v_i)\geq 5$, where $v_i\in N(v)$ with $d_{G[N(v)]}(v_i)=0$. We can obtain that $2|E(G)|\geq 7n-23$ by Subcases $4.1$-$4.2$ of $k\geq6$, and thus $|E(G)|\geq \lceil \frac{7n-23}{2} \rceil>\lceil \frac{17n-32}{6} \rceil$ by $n\geq 10$, (\ref{eq4.3}) holds.

If $|E(G[N(v)])|=3$, then $G[N(v)]\in\{K_{1,3}\cup K_1,K_3\cup 2K_1,P_4\cup K_1,P_3\cup P_2\}$. When $G[N(v)]\in\{K_{1,3}\cup K_1,K_3\cup 2K_1,P_4\cup K_1\}$, then $n\geq10$. Similar to Subcase $4.3$ of $k\geq6$ and taking $k=5$, we have $2|E(G)|\geq 7n-23$, (\ref{eq4.3}) holds. When $G[N(v)]\cong P_3\cup P_2$, then $n\geq9$. If $n=9$, then $|U_5|=3$, or $|U_5|=2$ with $|U_4|=1$ by $\delta(G)=5$, which implies $|U_3|+2|U_4|+3|U_5|\geq8$, thus $2|E(G)|\geq 7n-18$ by (\ref{eq4.9}), (\ref{eq4.3}) holds. If $n\geq10$, then $2|E(G)|\geq 7n-23$ similar to Subcase $4.3$ of $k\geq6$, (\ref{eq4.3}) holds.

If $|E(G[N(v)])|=4$, then $G[N(v)]\not\cong K_{1,4}$ by Claim $3$, and $\delta(G[N(v)])\leq 1$, which implies $n\geq9$ by $\delta(G)=5$. When $n\geq10$, similar to Case $4$ of $|E(G[N(v)])|=4$ with $k\geq6$, we have $2|E(G)|\geq 7n-22$, (\ref{eq4.3}) holds. When $n=9$, then $\delta(G[N(v)])=1$ and $G\in\{K_3\cup K_2,P_5,K_{1,3_{+}}\}$, where $K_{1,3_{+}}$ is a graph obtained by adding a pendant edge to a vertex of degree $1$ in $K_{1,3}$, and thus $|U_3|=|U_4|=|U_5|=1$ or $|U_4|=3$, which implies $|U_3|+2|U_4|+3|U_5|\geq6$, and $2|E(G)|\geq 7n-18$, (\ref{eq4.3}) holds.

If $|E(G[N(v)])|=5$, then $\delta(G[N(v)])\leq 2$. When $\delta(G[N(v)])\leq1$, we have $n\geq9$.
When $\delta(G[N(v)])=2$, we have $G[N(v)]\cong C_5$ and $n\geq8$.
If $n=8$, then $|U_5|=2$, but $G+v_iv_j$ contains no copy of $K_2\vee P_5$, where $v_iv_j\notin E(G)$, a contradiction. If $n=9$, then $U=U_3\cup U_4\cup U_5$, thus $2|E(G)|\geq 7n-19$, (\ref{eq4.3}) holds. If $n\geq10$, then $2|E(G)|\geq 7n-22$, (\ref{eq4.3}) holds.

If $|E(G[N(v)])|=6$, then $\delta(G[N(v)])\leq 3$. When $\delta(G[N(v)])\leq 1$, then $n\geq 9$, and thus $|E(G)|\geq \lceil \frac{7n-20}{2}\rceil>\lceil \frac{17n-32}{6}\rceil$. When $\delta(G[N(v)])=2$, then $n\geq 8$ and $|U_3|+|U_4|+|U_5|\geq1$ since there exists a copy of $K_{2}\vee P_{5}$ in $G+v_iv_j$, where $v_iv_j\notin E(G)$, and thus $|E(G)|\geq \lceil \frac{7n-19}{2}\rceil>\lceil \frac{17n-32}{6}\rceil$. Therefore, (\ref{eq4.3}) holds.

If $|E(G[N(v)])|\geq 7$, then $2|E(G)|\geq 7n-18$ by (\ref{eq4.9}), thus $|E(G)|\geq \lceil \frac{7n-18}{2} \rceil>\lceil \frac{17n-32}{6} \rceil$, (\ref{eq4.3}) holds.

Based on above arguments, we complete the proof of Theorem \ref{thm1-1}.
{\hfill $\square$ \par}

\section{Further research and problems}
In this section, we study the minimal $P_{k}$-saturated graphs for $k\in\{3,4\}$, and propose two problems worthy of further research.

For $k\in\{1,2\}$, $K_{s}\vee P_{k}\cong K_{s+k}$, and its saturation number has been determined in \cite{Erdos}.
For $k\geq3$, the saturation numbers of $K_1\vee P_k$ and $K_2\vee P_k$ have been determined in \cite{Hu2025} and Theorem \ref{thm3-3}, and Theorem \ref{thm1-1}, respectively.
For $k\geq3$ and $s\geq3$, we have the following result.

\begin{prop}\label{prop5-4}
For $n\geq k+s$ and $k,s\geq3$, $sat(n,K_{s}\vee P_{k})\leq \binom{s}{2}+s(n-s)+sat(n-s,P_{k})$.
\end{prop}

\begin{proof}
Observed that $K_{s}\vee P_{k}=K_{1}\vee(K_{s-1}\vee P_{k})=K_{1}\vee K_{1}\vee(K_{s-2}\vee P_{k})=\cdots=\underbrace{K_{1}\vee K_{1}\vee\cdots\vee K_{1}}_{s}\vee P_{k}$. By Corollary \ref{cor2-4}, we have
\vspace{-0.6cm}
\begin{align*}
sat(n,K_{s}\vee P_{k})
&\leq n-1+sat(n-1,K_{s-1}\vee P_{k})\\
&=n-1+n-2+sat(n-2,K_{s-2}\vee P_{k})\\
  &=\cdots\\
  &=\binom{s}{2}+s(n-s)+sat(n-s,P_{k}).
\qedhere
\end{align*}
\end{proof}

Based on the results of $k\in\{1,2\}$, $s=1$, $s=2$ and Proposition \ref{prop5-4}, we can propose the following conjecture.

\begin{conj}
Let $a_{k}$ be defined as in $\mathrm{(\ref{eq1.1})}$, $k,s\geq3$ and $n\geq a_k+s$. Then $sat(n, K_{s}\vee P_{k})=\binom{s}{2}+s(n-s)+sat(n-s,P_{k})$.
\end{conj}

By Theorems \ref{thm1-1}, \ref{thm2-2} and \ref{thm3-3}, the characterizations of the minimal $K_2\vee P_k$-saturated graphs and the minimal $K_1\vee P_k$-saturated graphs are highly dependent on the characterization of minimal $P_{k}$-saturated graphs.
In \cite{Kaszonyi}, Theorem \ref{thm2-1} does not provide a complete characterization of the minimal $P_{k}$-saturated graphs, but show that all minimal $P_{k}$-saturated trees for $k\geq5$ have a common structure. Thus the minimal $P_{k}$-saturated graphs are worth studying.
We study the minimal $P_{k}$-saturated graphs for $k\in\{3,4\}$, obtain the following propositions.

\begin{prop}\label{prop5-6}
Let $G$ be a graph of order $n\geq3$. Then $G$ is a minimal $P_{3}$-saturated graph if and only if $G\cong\lfloor \frac{n}{2} \rfloor K_2\cup (n-2\lfloor \frac{n}{2} \rfloor)K_1$.
\end{prop}

\begin{proof}
If $G\cong\lfloor \frac{n}{2} \rfloor K_2\cup (n-2\lfloor \frac{n}{2} \rfloor)K_1$, then $G$ contains no copy of $P_3$. For any $e\in E(\overline{G})$, $G+e$ has a copy of $P_3$, then $G$ is a $P_3$-saturated graph, and thus $G$ is a minimal $P_{3}$-saturated graph by $|E(G)|=\lfloor \frac{n}{2} \rfloor$ and (i) of Theorem \ref{thm2-1}.

If $G$ is a minimal $P_{3}$-saturated graph, then the number of vertices in each component of $G$ is less than 3. Otherwise, $G$ contains a copy of $P_3$, a contradiction. Clearly, $G$ has at most one isolated vertex. If not, let $u,v$ be two isolated vertices. Then $G+uv$ contains no copy of $P_3$, a contradiction. Therefore, $G\cong\lfloor \frac{n}{2} \rfloor K_2\cup (n-2\lfloor \frac{n}{2} \rfloor)K_1$.
\end{proof}

\begin{prop}\label{prop5-7}
Let $G$ be a graph of order $n\geq4$. Then $G$ is a minimal $P_{4}$-saturated graph if and only if $G\cong \frac{n}{2}K_2$ for even $n$, $G\cong K_3\cup\frac{n-3}{2}K_2$ or $G\cong K_{1,4}\cup\frac{n-5}{2}K_2$ for odd $n$.
\end{prop}

\begin{proof}
If $n$ is even and $G\cong \frac{n}{2}K_2$, then $G$ contains no copy of $P_4$. For any $e\in E(\overline{G})$, $G+e$ has a copy of $P_4$, then $G$ is a $P_4$-saturated graph, and thus $G$ is a minimal $P_{4}$-saturated graph by $|E(G)|=\frac{n}{2}$ and (ii) of Theorem \ref{thm2-1}. Similarly, if $n$ is odd, $G\cong K_3\cup\frac{n-3}{2}K_2$ or $G\cong K_{1,4}\cup\frac{n-5}{2}K_2$, then $G$ is a $P_4$-saturated graph, and thus $G$ is a minimal $P_{4}$-saturated graph by $|E(G)|=\frac{n+3}{2}$ and (ii) of Theorem \ref{thm2-1}.

Let $G$ be a minimal $P_{4}$-saturated graph of order $n$, $G_1,G_2,\ldots,G_t$ be  all the components of $G$. Then $|E(G)|=sat(n,P_4)$. We claim that $G$ has no isolated vertices. Otherwise, if $G$ has a isolated vertex, say $u$, then $t\geq2$. Without loss of generality, suppose $V(G_1)=\{u\}$. Since $G+uv$ contains a copy of $P_4$ for any $v\in V(G)\backslash \{u\}$, it follows that $|V(G_{i})|\geq 3$ for any $i\in\{2,3,\ldots,t\}$. Let $w_1\in V(G_i)$ for $i\in\{2,3,\ldots,t\}$, and $uw_1w_2w_3$ be a copy of $P_4$ in $G+uw_1$. We know $G+uw_2$ also contains a copy of $P_4$, say $P$. If $w_1$ or $w_3$ is not in $P$, then $P-\{u\}+w_1w_2$ or $P-\{u\}+w_2w_3$ is a copy of $P_4$ in $G_i$, a contradiction. Thus $w_1,w_3\in V(P)$ and $w_1w_3\in E(G)$, which implies $G[\{w_1,w_2,w_3\}]\cong K_3$. Then $|V(G_{i})|=3$ since $G_i$ contains no copy of $P_4$, and thus $G_{i}\cong K_3$ for any $i\in\{2,3,\ldots,t\}$, which implies $n=3t-2$, $|E(G)|=3t-3=n-1>sat(n,P_4)$, a contradiction.
Therefore, $G$ has no isolated vertices, say, $\delta(G)\geq1$.

If $n$ is even, then $|E(G)|=\frac{1}{2}\sum_{v\in V(G)}d(v)\geq\frac{n}{2}\delta(G)\geq\frac{n}{2}$, the equality holds if and only if $d(v)=1$ for any $v\in V(G)$. On the other hand, by (ii) of Theorem \ref{thm2-1}, we know $|E(G)|=\frac{n}{2}$, and then $G\cong \frac{n}{2}K_2$.

If $n$ is odd, then there exists a component with odd vertices. Without loss of generality, let $|V(G_1)|$ be odd. Then $|V(G_1)|\geq3$ and $\sum_{i=2}^{t}|V(G_i)|$ is even. If $t=2$, then $G_2\cong K_2$. If $t\geq 3$, then $G_2\cup G_3\cdots\cup G_t$ is $P_{4}$-saturated graph. Since $G$ is a minimal $P_{4}$-saturated graph, $G_2\cup G_3\cup\cdots\cup G_t$ is also a minimal $P_{4}$-saturated graph, and so $G_2\cong G_3\cup\cong\cdots\cong G_t\cong K_2$ since $\sum_{i=2}^{t}|V(G_i)|$ is even.
If $|V(G_1)|=3$, then $G_1\cong K_3$, and $G\cong K_3\cup\frac{n-3}{2}K_2$.
If $|V(G_1)|\geq 5$, then $diam(G_1)=2$ since $G_1$ contains no copy of $P_4$, and thus $G_1\cong K_{1,|V(G_1)|-1}$. Therefore, $|E(G)|=\frac{1}{2}\sum_{v\in V(G)}d(v)=\frac{1}{2}(|V(G_1)|-1+n-1)\geq\frac{n+3}{2}$, the equality holds if and only if $|V(G_1)|=5$, which implies $G\cong K_{1,4}\cup\frac{n-5}{2}K_2$.
\end{proof}

For $k\geq 5$, the minimal $P_{k}$-saturated graphs are complicated, and there exist many non-isomorphic graphs. Therefore, the minimal $P_{k}$-saturated graphs for $k\geq 5$ are worth studying, then we propose the following problem for further research.

\begin{pro}
Determine all minimal $P_{k}$-saturated graphs for $k\geq 5$.
\end{pro}

\section{Conclusion}
In this paper, by using different methods from Theorem \ref{thm2-2}, we show that $sat(n, K_{2}\vee P_{k})=2n-3+sat(n-2,P_{k})$ for $n\geq a_k+2$ and $k\geq 3$. Combining Theorem \ref{thm2-2}, we obtain that $sat(n, K_{1}\vee P_{k})=n-1+sat(n-1,P_{k})$ for $n\geq a_k+1$ and $k\geq 3$, which improves the bounds on $n$ in Theorem \ref{thm2-2}. Moreover, we characterize the minimal $K_{2}\vee P_{k}$-saturated graphs for $k\geq 3$, the minimal $K_{1}\vee P_{k}$-saturated graphs for $3\leq k\leq5$ and the minimal $P_{k}$-saturated graphs for $3\leq k\leq4$.

A \emph{book} $B_p$ is a union of $p$ triangles sharing one edge. A \emph{generalized book} $B_{b,p}$ is the union of $p$ copies of $K_{b+1}$ sharing a common $K_{b}$. The saturation numbers of a book and a generalized book have been determined by Chen, Faudree and Gould \cite{ChenandFaudree}. In \cite{Faudree2013}, Faudree and Gould determined the saturation number of the graph $K_t-P_4$.
The results are as follows.

\begin{prop}[\rm\!\!\cite{ChenandFaudree}]\label{prop5-1}
For $n\geq p^3+p$, $sat(n,B_p)= \frac{1}{2}((p+1)(n-1)-\lfloor \frac{p}{2} \rfloor \lfloor \frac{p}{2} \rfloor + \theta(n, p))$, where $\theta(n, p)=1$ if $p \equiv n-p/2 \equiv 0(\bmod 2)$, and $0$ otherwise.
\end{prop}

\begin{prop}[\rm\!\!\cite{ChenandFaudree}]\label{prop5-2}
For $n\geq 4(p+2b)^b$, $sat(n,B_{b,p})=\frac{1}{2}((p+2b-3)(n_b+1)-\lceil\frac{p}{2}\rceil \lfloor\frac{p}{2}\rfloor+\theta(n,p,b)+(b-1)(b-2)$, where
$\theta(n,p,b)=
\begin{cases}
1,~\text{if } p\equiv n-p/2-b \equiv0(\bmod 2), \\
0,~\text{otherwise}.
\end{cases}$
\end{prop}

\begin{prop}[\rm\!\!\cite{Faudree2013}]\label{prop5-3}
For $t\geq 5$ and $n\geq 7t-18$, $sat(n,K_t-P_4)=\lfloor \frac{(2t-7)(n-t+4)}{2} \rfloor+\binom{t-4}{2}+\eta(n,t)$, where $\eta(n,t)=2$ if $n-t$ is odd, and $0$ otherwise.
\end{prop}

By $B_2\cong K_1\vee P_3$, $B_{3,2}\cong K_{2}\vee P_{3}$, $K_5-P_4\cong K_1\vee P_4$, Propositions \ref{prop5-1}-\ref{prop5-3}, Theorems \ref{thm1-1} and \ref{thm3-3}, the following table is obtained, and the results from Theorems \ref{thm1-1} and \ref{thm3-3} improve the results from Propositions \ref{prop5-1}-\ref{prop5-3}. Moreover, Theorems \ref{thm1-1} and \ref{thm3-3} completely
characterize the minimal $K_2\vee P_3$-saturated graphs, $K_1\vee P_3$-saturated graphs and $K_1\vee P_4$-saturated graphs, respectively.

\begin{table}[htbp]
    \centering
    \small
    \caption{Some results on the saturation number of special graphs.}
    \label{tab1}
    \renewcommand{\arraystretch}{1.5}
    \begin{tabular}{|c|c|c|c|}
        \hline
        & $K_1\vee P_3$ & $K_2\vee P_3$ & $K_1\vee P_4$\\
        \hline
        Propositions \ref{prop5-1}-\ref{prop5-3} & \multicolumn{1}{l|}{$\lfloor \frac{3n-3}{2} \rfloor$, $n\geq 10$} & \multicolumn{1}{l|}{$\lfloor \frac{5n-8}{2} \rfloor$, $n\geq 4\times 8^3$}  & \multicolumn{1}{l|}{$\begin{cases}
\frac{3n}{2},  & \text{even $n(\geq 18)$}\\
\frac{3n-3}{2}, & \text{odd $n(\geq 17)$}
\end{cases}$}\\
        \hline
        \makecell{Theorems \ref{thm1-1} and \ref{thm3-3}} & \multicolumn{1}{l|}{$\lfloor \frac{3n-3}{2} \rfloor$, $n\geq 4$} & \multicolumn{1}{l|}{$\lfloor \frac{5n-8}{2} \rfloor$, $n\geq 5$} & \multicolumn{1}{l|}{$\begin{cases}
\frac{3n}{2},  & \text{even $n(\geq 4)$}\\
\frac{3n-3}{2}, & \text{odd  $n(\geq 5)$}
\end{cases}$}\\
        \hline
    \end{tabular}
     \renewcommand{\arraystretch}{1}
\end{table}

\section*{Appendix A. Proofs}
We present the proofs of Lemma \ref{lem3-4} and Theorem \ref{thm3-3} here.

\subsection*{A.1. Proof of Lemma \ref{lem3-4}}

\begin{proof}
If $diam(G)\neq 2$, then $G$ is not $K_{1}\vee P_{4}$-saturated by Proposition \ref{pro3-2*}.

If $diam(G)=2$, then $5\leq n\leq 10$ by Observation \ref{obs1}. Since $\sum_{v\in V(G)}d(v)$ is even and $d(v)=3$ for any $v\in V(G)$, we have $n\in\{6,8,10\}.$ For $v\in V(G)$, we take $N(v)=\{x,y,z\}$ and $V(G)\backslash N[v]=\{w_1,w_2,\ldots,w_{n-4}\}$.

Let $n=6$. If $w_1w_2\notin E(G)$, then $N(w_1)=N(w_2)=\{x,y,z\}$ by Observation \ref{obs1}. Thus $G$ is not $K_{1}\vee P_{4}$-saturated since $G+vw_1$ contains no copy of $K_{1}\vee P_{4}$. If $w_1w_2\in E(G)$, then $N(w_1)\backslash\{w_2\}\subset\{x,y,z\}$ by Observation \ref{obs1}. Without loss of generality, let $N(w_1)\backslash\{w_2\}=\{x,y\}$. Then $w_2z\in E(G)$. Otherwise, $N(w_2)=\{w_1,x,y\}$, and thus $d(z)=1$, a contradiction. Since $G$ is $3$-regular, we have $G\cong H_4$.

Let $n=8$. Then $0\leq|E(G[N(v)])|\leq1$.

If $|E(G[N(v)])|=1$, without loss of generality, let $xy\in E(G)$, $w_1\in N(x)$, $w_2\in N(y)$ and $w_3,w_4\in N(z)$. Then $w_iw_1,w_iw_2\in E(G)$ by $d(w_i,x)=d(w_i,y)=2$ for $i\in \{3,4\}$ and Observation \ref{obs1}. Since $G+vw_1$ contains no copy of $K_{1}\vee P_{4}$, it follows that $G$ is not $K_{1}\vee P_{4}$-saturated.

If $|E(G[N(v)])|=0$, then $(N(x)\cap N(y))\backslash \{v\}\neq\emptyset$ or $(N(y)\cap N(z))\backslash \{v\}\neq\emptyset$ or $(N(x)\cap N(z))\backslash \{v\}\neq\emptyset$. Without loss of generality, let $w_1\in(N(x)\cap N(y))\backslash \{v\}\neq\emptyset$. When $w_1z\in E(G)$, then $G+vw_1$ contains no copy of $K_{1}\vee P_{4}$, and thus $G$ is not $K_{1}\vee P_{4}$-saturated.
When $w_1z\notin E(G)$, we take $w_2\in N(w_1)\backslash \{x,y\}$. Then $w_2z\in E(G)$ by $d(w_1,z)=2$. If $w_2\in N(x)$, then there exists $w_3\in N(y)$ by $d(x)=d(y)=3$, and thus there exists $w_4\in N(z)\cap N(w_3)$, which implies $d(w_4)=2$, a contradiction. Therefore, $w_2\notin N(x)$, and $w_2\notin N(y)$ similarly, which implies that $G+vw_1$ contains no copy of $K_1\vee P_4$, and thus $G$ is not $K_{1}\vee P_{4}$-saturated.

Let $n=10$. Then $|E(G[N(v)])|=0$. Let $w_1,w_2\in N(x)$, $w_3,w_4\in N(y)$ and $w_5,w_6\in N(z)$.
There exist $i\in\{3,4\}$ and $j\in\{5,6\}$ such that $w_1w_i,w_1w_j\in E(G)$ by $d(w_1,y)=d(w_1,z)=2$ and Observation \ref{obs1}, and thus $G$ is not $K_{1}\vee P_{4}$-saturated
since $G+vw_1$ contains no copy of $K_{1}\vee P_{4}$.
\end{proof}

\subsection*{A.2. Proof of Theorem \ref{thm3-3}}
\begin{proof}
Let $F$ be a minimal $P_{k}$-saturated graph of order $n-1$ and $G\in\mathbb{G}$. Then $K_1\vee F$ and $G$ are $K_{1}\vee P_{k}$-saturated graphs for $k\geq3$ by Lemma \ref{lem2-3} and direct checking.

By Theorem \ref{thm2-2}, we only consider $k\in\{3,4\}$, and $k=5$ with $n=6$.
Let $G$ be a $K_{1}\vee P_{k}$-saturated graph of order $n$.
In order to show (\ref{eq3.1}), by Theorem \ref{thm2-1}, we only need to show
\begin{equation}\label{eq3.2}
|E(G)|\geq n-1+sat(n-1,P_{k}).
\end{equation}

If $\Delta(G)=n-1$, then (\ref{eq3.2}) holds, with the equality if and only if $G\cong K_{1}\vee F$ by Proposition \ref{pro3-1*}, where $F$ is a minimal $P_{k}$-saturated graph of order $n-1$.

Now we assume $\Delta(G)\leq n-2$. Then $diam(G)=2$ by Proposition \ref{pro3-2*}, and thus $\delta(G)\geq2$.
In the following, we divide the remaining proof into three cases based on the values of $k$.

$\mathbf{Case~1:}$
$k=5$.

By Theorems \ref{thm2-1} and \ref{thm2-2}, we only need to show $|E(G)|>5+sat(5,P_5)=9$.

For any $e\in E(\overline{G})$, since $G+e$ contains a copy of $K_{1}\vee P_{5}$, we have $\Delta(G+e)=5$, and thus $\Delta(G)=4$ by $\Delta(G)\leq 4$.
Let $v\in V(G)$ with $d_{G}(v)=4$, $N(v)=\{v_1,v_2,v_3,v_4\}$ and $V(G)\backslash N[v]=\{w\}$ with $2\leq d_{G}(w)\leq 4$.
Since $G+vw$ contains a copy of $K_{1}\vee P_{5}$, say $Q$, we have $d_{Q}(v)=5$ or $d_{Q}(w)=5$, and thus $|E(G[N(v)])|\geq 2$ by $Q-v\cong P_5$ or $Q-w\cong P_5$.

If $d_{G}(w)=4$, then $|E(G)|\geq d_{G}(v)+d_{G}(w)+|E(G[N(v)])|\geq4+4+2=10$.

If $d_{G}(w)<4$, then $d_{G}(v_i)=4$ since $G+wv_i$ contains a copy of $K_1\vee P_5$ for any $v_i\notin N(w)$ and $1\leq i\leq4$, and thus
\[|E(G)|\geq
\begin{cases}
d_{G}(w)+d_{G}(v)+d_{G}(v_i)-1=10,  & \text{if $d_{G}(w)=3$ with $v_i\notin N(w)$}; \\
d_{G}(w)+d_{G}(v)+d_{G}(v_i)+d_{G}(v_j)-3=11, & \text{if $d_{G}(w)=2$ with $v_i,v_j\notin N(w)$}.
\end{cases}\]


$\mathbf{Case~2:}$ $k=3$.

In this case, we show that (\ref{eq3.2}) holds, with equality if and only if $G\in\{H_1,H_2,H_3\}$.

If $\delta(G)\geq 3$, then $|E(G)|\geq \frac{n\delta(G)}{2}=\frac{3n}{2}>\lfloor \frac{3n-3}{2} \rfloor$, (\ref{eq3.2}) holds.

Now we consider $\delta(G)=2$. Let $v\in V(G)$, $N(v)=\{v_1,v_2\}$, $U=V(G)\backslash N[v]$, $Z=N(v_1)\cap N(v_2)\cap U$, $X=(N(v_1)\cap U)\backslash Z$ and $Y=(N(v_2)\cap U)\backslash Z$.

$\mathbf{Subcase~2.1:}$ $v_1v_2\in E(G)$.

Then $Z=\emptyset$. Otherwise, if $z\in Z$, then $G[\{v,v_1,v_2,z\}]$ is a copy of $K_1\vee P_3$, a contradiction.
On the other hand, $X\neq\emptyset$ and $Y\neq\emptyset$ by $\Delta(G)\leq n-2$.

If $d(u)\geq3$ for some $u\in U$, then $2|E(G)|=\sum_{u\in N[v]}d(u)+\sum_{u\in U}d(u)\geq d(v)+(2+|X|)+(2+|Y|)+(3+2(n-4))= 3n-2$, and thus $|E(G)|\geq\lceil \frac{3n-2}{2} \rceil >\lfloor \frac{3n-3}{2} \rfloor$, (\ref{eq3.2}) holds.

If $d(u)=2$ for any $u\in U$, then $|X|=|Y|=1$, and thus $n=5$, $G\cong H_3$. Otherwise, without loss of generality, let $|X|\geq2$, $x_1,x_2\in X$ and $y\in Y$. Clearly, $E[X,Y]\neq \emptyset$ by $diam(G)=2$ and we take $x_1y\in E(G)$. Thus $d(x_1)\geq3$ or $d(y)\geq 3$ by $d(x_2,y)\leq 2$, a contradiction. Therefore, $|E(G)|=6=\lfloor \frac{3n-3}{2} \rfloor$, (\ref{eq3.2}) holds.

$\mathbf{Subcase~2.2:}$ $v_1v_2\notin E(G)$.

Then $Z\neq\emptyset$. Otherwise, $G+v_1v_2$ contains no copy of $K_{1}\vee P_{3}$, a contradiction. Moreover, $Z$ is an independent set of $G$ since $G$ contains no copy of $K_{1}\vee P_{3}$.

If $X=Y=\emptyset$, then $G[U]\cong (n-3)K_1$, and thus $|E(G)|=2(n-2)\geq\lfloor \frac{3n-3}{2} \rfloor$, (\ref{eq3.2}) holds, with equality if and only if $G\cong H_1$ for $n=4$, or $G\cong H_2$ for $n=5$.

If exactly one of $X$ and $Y$ is empty, without loss of generality, let $Y=\emptyset$ and $X\neq\emptyset$. Then $n\geq 5$.
For any $x\in X$, there exists a unique vertex $z_{x}\in Z$ such that $xz_{x}\in E(G)$ by $d(x,v_2)=2$. Otherwise, $|N(x)\cap Z|\geq2$, and $G[\{v_1,x\}\cup (N(x)\cap Z)]$ contains a copy of $K_1\vee P_3$, a contradiction.
When $|X|\geq 2$, say, $x_1,x_2\in X$, then $z_{x_1}\neq z_{x_2}$. Otherwise, $G[\{v_1,x_1,x_2,z_{x_1}\}]$ contains a copy of $K_1\vee P_3$, a contradiction. In addition, $x_1x_2\notin E(G)$. Otherwise, $G[\{v_1,x_1,x_2,z_{x_1}\}]$ contains a copy of $K_1\vee P_3$, a contradiction. Therefore, $X$ is an independent set of $G$, and $G[U]\cong |X|K_2\cup(n-3-2|X|)K_1$. Thus $|E(G)|=2+|X|+2|Z|+|X|=2(n-2)\geq\lfloor \frac{3n-3}{2} \rfloor$, (\ref{eq3.2}) holds, with equality if and only if $G\cong H_3$ for $n=5$.

If $X\neq\emptyset$ and $Y\neq\emptyset$, then $n\geq 6$. For any $x\in X$, there exists either a vertex $z\in Z$ such that $xz\in E(G)$, or a vertex $y\in Y$ such that $xy\in E(G)$ by $d(x,v_2)=2$.
Clearly, $G[X\cup Z]$ contains no copy of $P_3$.
When $xz\in E(G)$, then $d(z)\geq 4$, or $d(z)\geq 3$, $d(x)\geq3$ by $d(x,y_1)\leq2$ with some $y_1\in Y$.
When $xy\in E(G)$, then $d(x)\geq 3$ and $d(y)\geq3$ since $G+vx$ and $G+vy$ must contain a copy of $K_1\vee P_3$.
Thus there are two vertices with degree at least $3$ or one vertex with degree at least $4$ in $U$. Therefore, $2|E(G)|=\sum_{u\in N[v]}d(u)+\sum_{u\in U}d(u)\geq\left[2+(1+|X|+|Z|)+(1+|Y|+|Z|)\right]+2n-4=3n-3+|Z|\geq3n-2$, and thus $|E(G)|\geq\lceil \frac{3n-2}{2}\rceil >\lfloor \frac{3n-3}{2} \rfloor$, (\ref{eq3.2}) holds.

$\mathbf{Case~3:}$ $k=4$.

In this case, we show that (\ref{eq3.2}) holds, say, $|E(G)|\geq\lfloor \frac{3(n-1)}{2} \rfloor+\sigma_{n}$, with equality if and only if $G\in\{H_4,H_5,H_6\}$, where $\sigma_{n}=2$ if $n$ is even, and $0$ otherwise.

If $\delta(G)\geq 3$, then $|E(G)|\geq \frac{3n}{2}\geq\lfloor \frac{3(n-1)}{2} \rfloor+\sigma_{n}$, with equality if and only if $G\cong H_4$ for $n=6$ by Lemma \ref{lem3-4}.

Next, we assume $\delta(G)=2$. Let $v\in V(G)$, $N(v)=\{v_1,v_2\}$, $U=V(G)\backslash N[v]$, $Z=N(v_1)\cap N(v_2)\cap U$, $X=(N(v_1)\cap U)\backslash Z$ and $Y=(N(v_2)\cap U)\backslash Z$.

$\mathbf{Subcase~3.1:}$
$v_1v_2\in E(G)$.

Clearly, $X\neq\emptyset$, $Y\neq\emptyset$ by $\Delta(G)\leq n-2$, $d(v)+d(v_1)+d(v_2)=n+3+|Z|$, and
\begin{equation}\label{eq3.3}
2|E(G)|=\sum_{u\in N[v]}d(u)+\sum_{u\in U}d(u)=n+3+|Z|+\sum_{u\in U}d(u).
\end{equation}

$\mathbf{Subcase~3.1.1:}$ $Z=\emptyset$.

For any $x\in X$, $y\in Y$, both $G+vx$ and $G+vy$ contain a copy of $K_1\vee P_4$, then
neither $X$ nor $Y$ is an independent set.
Let $x_1x_2\in E(G[X])$ and $y_1y_2\in E(G[Y])$. Then $d(x_1)+d(x_2)+d(y_1)+d(y_2)\geq 12$ by $d(x_i,y_j)\leq 2$ for $i,j\in\{1,2\}$. Thus $2|E(G)|\geq n+3+12+2(n-7)=3n+1>2(\lfloor \frac{3(n-1)}{2} \rfloor+\sigma_{n})$ by (\ref{eq3.3}), (\ref{eq3.2}) holds.

$\mathbf{Subcase~3.1.2:}$ $Z\neq\emptyset$.

For any $z\in Z$, $d(z)=2$ since $G$ contains no copy of $K_{1}\vee P_{4}$, and thus $Z$ is an independent set of $G$. Then $|Z|\geq 2$. Otherwise, $G+vz$ contains no copy of $K_{1}\vee P_{4}$, where $Z=\{z\}$, a contradiction.

If $|Z|\geq 3$, then $2|E(G)|\geq n+3+|Z|+2(n-3)\geq 3n$, and thus $|E(G)|\geq \lceil \frac{3n}{2}\rceil\geq\lfloor \frac{3(n-1)}{2} \rfloor+\sigma_{n}$, (\ref{eq3.2}) holds, with equality if and only if $n$ is even, $|Z|=3$, and $d(u)=2$ for any $u\in U$, which implies $|X|=|Y|=1$, and $G\cong H_6$ by $diam(G)=2$.

If $|Z|=2$, then $n\geq 7$. For any $x\in X$ and $y\in Y$, there exists a path with length $1$ or $2$ between $x$ and $y$ in $G[X\cup Y]$ by $d(x,y)\leq 2$. Then $G[X\cup Y]$ is connected, and $|E(G[X\cup Y])|\geq |X|+|Y|-1=n-6$. When $n\geq 9$,
then $|E(G)|\geq 3+|X|+|Y|+2|Z|+|E(G[X\cup Y])|\geq 2n-4>\lfloor \frac{3(n-1)}{2} \rfloor+\sigma_{n}$, (\ref{eq3.2}) holds.
When $n=7$, then $|E(G)|=10>\lfloor \frac{3(n-1)}{2} \rfloor+\sigma_{n}=9$, (\ref{eq3.2}) holds.
When $n=8$, then $G[X\cup Y]\cong K_3$. Otherwise, there exist $w_1,w_2\in X\cup Y$ with $w_1w_2\notin E(G)$, but $G+w_1w_2$ contains no copy of $K_1\vee P_4$, a contradiction.
Thus $|E(G)|=13>\lfloor \frac{3(n-1)}{2} \rfloor+\sigma_{n}=12$, (\ref{eq3.2}) holds.

$\mathbf{Subcase~3.2:}$
$v_1v_2\notin E(G)$.

Then $|Z|\geq 1$. Otherwise, $G+v_1v_2$ contains no copy of $K_{1}\vee P_{4}$, a contradiction. For any $z\in Z$, since $d(v)=2$ and $G+vz$ contains a copy of $K_{1}\vee P_{4}$, it follows that $d(z)\geq 3$.

Clearly, we have $d(v)+d(v_1)+d(v_2)=n+1+|Z|$, and
\begin{equation}\label{eq3.4}
2|E(G)|=\sum_{u\in N[v]}d(u)+\sum_{u\in U}d(u)=n+1+|Z|+\sum_{u\in U}d(u).
\end{equation}

$\mathbf{Subcase~3.2.1:}$ $|Z|\geq 3$.

Then
$2|E(G)|\geq n+1+|Z|+3|Z|+2(|X|+|Y|)=3n-5+2|Z|\geq 3n+1>2(\lfloor\frac{3(n-1)}{2} \rfloor+\sigma_{n})$ by (\ref{eq3.4}) and $d(z)\geq3$ for any $z\in Z$, (\ref{eq3.2}) holds.

$\mathbf{Subcase~3.2.2:}$ $|Z|=2$.

We take $Z=\{z_1,z_2\}$. Then $d(z_1)+d(z_2)\geq 6$.

If $d(z_1)+d(z_2)\geq 8$, then $2|E(G)|\geq n+1+|Z|+8+2(n-5)=3n+1$ by (\ref{eq3.4}).

If $d(z_1)+d(z_2)=7$, then $z_1z_2\notin E(G)$. Otherwise, $G[N(z_1)\cup N(z_2)]$ contains a copy of $K_1\vee P_4$, a contradiction. Without loss of generality, we take $d(z_1)=4$ and $d(z_2)=3$. By Observation \ref{obs1} and $z_1z_2\notin E(G)$, we have $N(z_1)\backslash\{v_1,v_2\}\subseteq X\cup Y$.

We first consider $|N(z_1)\cap X|=2$ or $|N(z_1)\cap Y|=2$, without loss of generality, let $|N(z_1)\cap X|=2$ and $N(z_1)\cap X=\{x_1,x_2\}$. Then $d(x_1)\geq 3$ or $d(x_2)\geq 3$. Otherwise, $d(x_1)=d(x_2)=2$, and $G+x_1x_2$ contains no copy of $K_1\vee P_4$, a contradiction. Thus $2|E(G)|\geq n+1+|Z|+7+2(n-6)+3=3n+1$ by (\ref{eq3.4}).

Next we consider $|N(z_1)\cap X|=1$ and $|N(z_1)\cap Y|=1$. Let $N(z_1)\cap X=\{x_1\}$ and $N(z_1)\cap Y=\{y_1\}$. Then $d(x_1)\geq 3$ and $d(y_1)\geq 3 $ since $G+vx_1$ and $G+vy_1$ must contain a copy $K_1\vee P_4$. Thus $2|E(G)|\geq n+1+|Z|+7+2(n-7)+6=3n+2$ by (\ref{eq3.4}).

If $d(z_1)+d(z_2)=6$, then $d(z_1)=d(z_2)=3$ by $d(z)\geq3$ for any $z\in Z$.
When $z_1z_2\in E(G)$, then $n\neq 6$ by $\delta(G)=2$, and $d(w)\geq 3$ for any $w\in U$ since $G+vw$ must contain a copy of $K_1\vee P_4$, and thus $2|E(G)|\geq n+1+|Z|+3(n-3)=4n-6$ by (\ref{eq3.4}).
When $z_1z_2\notin E(G)$, let $N(z_1)\cap U=\{w_1\}$ and $N(z_2)\cap U=\{w_2\}$. For $w_1=w_2$, $d(w_1)\geq 4$ since $G+vw_1$ must contain a copy of $K_1\vee P_4$, and thus $2|E(G)|\geq n+1+|Z|+6+2(n-6)+4=3n+1$ by (\ref{eq3.4}).
For $w_1\neq w_2$, $d(w_i)\geq 3$ for $i\in\{1,2\}$ since $G+vw_i$ contains a copy of $K_1\vee P_4$, and thus $2|E(G)|\geq n+1+|Z|+6+2(n-7)+6=3n+1$ by (\ref{eq3.4}).

Combining above arguments, we have $|E(G)|>\lfloor \frac{3(n-1)}{2} \rfloor+\sigma_{n}$, (\ref{eq3.2}) holds.

$\mathbf{Subcase~3.2.3:}$ $|Z|=1$.

If $X=\emptyset$ or $Y=\emptyset$, without loss of generality, let $Y=\emptyset$, $X=\{x_1,x_2,\dots,x_{n-4}\}$ and $Z=\{z\}$. Then $x_iz\in E(G)$ by $d(x_i,v_2)=2$ for $1\leq i\leq n-4$, and thus $d(z)=n-2$.
When $n=5$, then $G+vx_1$ contains no copy of $K_1\vee P_4$, a contradiction.
When $n=6$, then $x_1x_2\in E(G)$. Otherwise, $G+x_1x_2$ contains no copy of $K_1\vee P_4$, a contradiction. Thus $G\cong H_5$, and $|E(G)|=9=\lfloor \frac{3(n-1)}{2} \rfloor+\sigma_{n}$, (\ref{eq3.2}) holds.
When $n\geq 7$, we have $G[X]\cong (n-4)K_1$ since $G$ contains no copy of $K_1\vee P_4$. Thus $|E(G)|=2n-4\geq\lfloor \frac{3(n-1)}{2} \rfloor+\sigma_{n}$, (\ref{eq3.2}) holds, with equality if and only if $G\cong H_6$ for $n=8$.

If $X\neq\emptyset$ and $Y\neq\emptyset$, then $n\geq6$ and we take $Z=\{z\}$. By $d(z)\geq 3$, we can complete the proof by the following arguments.

Let $d(z)=3$. Without loss of generality, we take $N(z)\backslash\{v_1,v_2\}=\{w\}$ with $w\in X$. Then $d(w)\geq3$ by $d(w,y)\leq2$, where $y\in Y$. Moreover, $|X|\geq 2$. Otherwise, $G+vw$ contains no copy of $K_1\vee P_4$, a contradiction. Thus $n\geq 7$.
For any $x\in X\backslash\{w\}$, $N(x)\cap Y\neq\emptyset$ by $d(x,v_2)=2$. Since $G+vx$ must contain a copy of $K_1\vee P_4$, we have $d(x)\geq3$.
Similarly, $d(y)\geq3$ for any $y\in Y$.
Thus $d(u)\geq 3$ for any $u\in U$. It follows that $2|E(G)|\geq n+1+|Z|+3(n-3)=4n-7>2(\lfloor \frac{3(n-1)}{2} \rfloor+\sigma_{n})$ by (\ref{eq3.4}), (\ref{eq3.2}) holds.

Let $d(z)\geq 4$. Then we consider $|N(z)\cap X|$ and $|N(z)\cap Y|$.

When $|N(z)\cap X|=1$ and $|N(z)\cap Y|=1$, then $d(z)=4$, and we take $N(z)\cap X=\{x_1\}$ and $N(z)\cap Y=\{y_1\}$. Then $x_1y_1\notin E(G)$ since $G$ contains no copy of $K_{1}\vee P_{4}$. If $n=6$, then $G+vx_1$ contains no copy of $K_{1}\vee P_{4}$, a contradiction.
If $n\geq 7$, without loss of generality, let $|X|\geq 2$ and $x_2\in X$. Then there exists some $y\in Y$ such that $x_2y\in E(G)$ by $d(x_2,v_2)=2$. Thus $d(x_i)\geq 3$ for $i\in\{1,2\}$ and $d(y_1)\geq 3$ since $G+vx_i$ and $G+vy_1$ must contain a copy of $K_{1}\vee P_{4}$. Therefore, $2|E(G)|\geq n+1+|Z|+d(z)+d(x_1)+d(x_2)+d(y_1)+2(n-7)\geq3n+1$ by (\ref{eq3.4}), (\ref{eq3.2}) holds.

When $|N(z)\cap X|\geq 2$ or $|N(z)\cap Y|\geq 2$, without loss of generality, let $x_1,x_2\in N(z)\cap X$. Now we study the case $x_1x_2\in E(G)$ or the case $x_1x_2\notin E(G)$.

Let $x_1x_2\in E(G)$. Then $G[\{x_1,x_2,z\}]\cong K_3$ is a connected component of $G[X\cup Z]$ since $G$ contains no copy of $K_{1}\vee P_{4}$. If there exists $y\in Y$ such that $yz\in E(G)$, then $d(z)+d(y)\geq 7$. If $yz\notin E(G)$ for any vertex $y\in Y$, then $yx_1\in E(G)$ or $yx_2\in E(G)$ by $d(y,x_1)\leq 2$, and thus $d(y)\geq 3$ since $G+vy$ contains a copy of $K_1\vee P_4$, and thus $d(z)+d(y)\geq 7$. Therefore, $2|E(G)|\geq n+1+|Z|+d(z)+d(y)+2(n-7)+d(x_1)+d(x_2)\geq3n+1$ by (\ref{eq3.4}), (\ref{eq3.2}) holds.

Let $x_1x_2\notin E(G)$.
Then $x_1$ and $x_2$ are isolated vertices in $G[X]$. Otherwise, there exists $x_3\in X$ with $x_ix_3\in E(G)$ for some $i\in \{1,2\}$, then $G[\{v_1,x_1,x_2,x_3,z\}]$ contains a copy of $K_1\vee P_4$, a contradiction.
If there exists $y\in Y$ such that $yz\notin E(G)$, then $d(x_i)\geq3$ by $d(x_i,y)\leq2$ for any $i\in\{1,2\}$. Since $G+vy$ must contain a copy of $K_1\vee P_4$, say $Q$, we have $d(y)\geq 3$ or $d(y_1)\geq3$ for some $y_1\in V(Q)$. Thus $2|E(G)|\geq n+1+|Z|+d(z)+2(n-7)+9\geq3n+1$ by (\ref{eq3.4}), (\ref{eq3.2}) holds.
If $yz\in E(G)$ for any $y\in Y$, then $xy\notin E(G)$ for any $x\in N(Z)\cap X$. Otherwise, $G[\{z,v_1,v_2,x,y\}]$ contains a copy of $K_1\vee P_4$, a contradiction. Thus $d(x_1)=d(x_2)=2$. Since $x_1x_2\notin E(G)$, $G+x_1x_2$ must contain a copy of $K_1\vee P_4$, say $Q'$. It is easy to obtain $\{x_1,x_2,z,v_1\}\subseteq V(Q')$, $d_{Q'}(v_1)=4$ or $d_{Q'}(z)=4$. Then $|N(z)\cap X|\geq3$, and thus $|Y|\geq2$ since $G+vy$ must contain a copy of $K_1\vee P_4$, where $y\in Y$. Thus $d(z)\geq7$ and $2|E(G)|\geq n+1+|Z|+d(z)+2(n-4)\geq3n+1$ by (\ref{eq3.4}), (\ref{eq3.2}) holds.

Based on above arguments and Theorem \ref{thm2-2}, we complete the proof of Theorem \ref{thm3-3}.
\end{proof}

\def\polhk#1{\setbox0=\hbox{#1}{\ooalign{\hidewidth
\lower1.5ex\hbox{`}\hidewidth\crcr\unhbox0}}}

\end{document}